\documentclass{amsart}
\usepackage[utf8]{inputenc}
\usepackage{todonotes} 
\usepackage{amsmath, amssymb, amsthm} 
\usepackage{latexsym} 
\usepackage{dsfont}
\usepackage[flushleft]{paralist}
\usepackage[all]{xy}
\usepackage{stmaryrd}
\usepackage{tikz-cd}
\usepackage{hyperref} 
\usepackage[OT2,T1]{fontenc}
\tikzcdset{scale cd/.style={every label/.append style={scale=#1},
    cells={nodes={scale=#1}}}}
\DeclareSymbolFont{cyrletters}{OT2}{wncyr}{m}{n}
\DeclareMathSymbol{\Zhe}{\mathalpha}{cyrletters}{"11} 
    
\begin{document}
\title{Fine Selmer groups and ideal class groups}
\author{Sören Kleine} 
\address{Institut für Theoretische Informatik, Mathematik und Operations Research, Universität der Bundeswehr München, Werner-Heisenberg-Weg 39, 85577 Neubiberg, Germany} 
\email{soeren.kleine@unibw.de} 
\author{Katharina Müller} 
\address{Department of Mathematics and Statistics, Université Laval,  Québec City, Canada} 
\email{katharina.mueller.1@ulaval.ca} 

\subjclass[2020]{11R23} 
\keywords{uniform $p$-adic Lie extension, fine Selmer group, abelian variety with complex multiplication, $\mu$-invariant, generalised Iwasawa invariants, weak Leopoldt conjecture}

\newcommand{\R}{\mathds{R}}
 	\newcommand{\Z}{\mathds{Z}}
 	\newcommand{\N}{\mathds{N}}
 	\newcommand{\Q}{\mathds{Q}}
 	\newcommand{\K}{\mathds{K}}
 	\newcommand{\M}{\mathds{M}} 
 	\newcommand{\C}{\mathds{C}}
 	\newcommand{\B}{\mathds{B}}
 	\newcommand{\LL}{\mathds{L}}
 	\newcommand{\F}{\mathds{F}}
 	\newcommand{\p}{\mathfrak{p}} 
 	\newcommand{\q}{\mathfrak{q}} 
 	\newcommand{\f}{\mathfrak{f}} 
 	\newcommand{\Pot}{\mathcal{P}}
 	\newcommand{\Gal}{\textup{Gal}}
 	\newcommand{\rg}{\textup{rank}}
 	\newcommand{\id}{\textup{id}}
 	\newcommand{\Ker}{\textup{Ker}}
 	\newcommand{\Image}{\textup{Im}} 
 	\newcommand{\pr}{\textup{pr}}
 	\newcommand{\la}{\langle}
 	\newcommand{\ra}{\rangle}
 	\newcommand{\gdw}{\Leftrightarrow}
 	\newcommand{\pfrac}[2]{\genfrac{(}{)}{}{}{#1}{#2}}
 	\newcommand{\Ok}{\mathcal{O}}
 	\newcommand{\Norm}{\mathrm{N}} 
 	\newcommand{\coker}{\mathrm{coker}}
 	\newcommand{\dotcup}{\stackrel{\textstyle .}{\bigcup}}
 	\newcommand{\Cl}{\mathrm{Cl}}
 	\newcommand{\Sel}{\textup{Sel}} 

 	\newtheorem{lemma}{Lemma}[section] 
 	\newtheorem{prop}[lemma]{Proposition} 
 	\newtheorem{defprop}[lemma]{Definition and Proposition} 
 	\newtheorem{conjecture}[lemma]{Conjecture} 
 	\newtheorem{thm}[lemma]{Theorem} 
 	\newtheorem*{thm*}{Theorem} 
 	\newtheorem{cor}[lemma]{Corollary}
 	\newtheorem{claim}[lemma]{Claim}

 	\theoremstyle{definition}
 	\newtheorem{def1}[lemma]{Definition} 
 	\newtheorem{ass}[lemma]{Assumption}
 	\newtheorem{rem}[lemma]{Remark} 
 	\newtheorem{rems}[lemma]{Remarks} 
 	\newtheorem{example}[lemma]{Example} 
 	\newtheorem{fact}[lemma]{Fact}

\maketitle

\begin{abstract} 
   Let $K$ be a number field, let $A$ be an abelian variety defined over $K$, and let $K_\infty/K$ be a uniform $p$-adic Lie extension. We compare several arithmetic invariants of Iwasawa modules of ideal class groups on the one side and fine Selmer groups of abelian varieties on the other side. 
   
   If $K_\infty$ contains sufficiently many $p$-power torsion points of $A$, then we can compare the ranks and the Iwasawa $\mu$-invariants of these modules over the Iwasawa algebra. 
   In several special cases (e.g. multiple $\Z_p$-extensions), we can also prove relations between suitably generalised Iwasawa $\lambda$-invariants of the two types of Iwasawa modules. 
   
   In the literature, different kinds of generalised Iwasawa $\lambda$-invariants have been introduced for ideal class groups and Selmer groups. We define analogues of both concepts for fine Selmer groups and compare the resulting invariants. 
   
   In order to obtain some of our main results, we prove new asymptotic formulas for the growth of ideal class groups and fine Selmer groups in multiple $\Z_p$-extensions. 
\end{abstract} 

\section{Introduction} \label{section:1} 

The origins of Iwasawa theory concerned the investigation of the growth of class numbers in $\Z_p$-extensions of number fields. Iwasawa proved the existence of an asymptotic formula 
\begin{align} \label{eq:iw-intro} 
  e_n = \mu(K_\infty/K) \cdot p^n + \lambda(K_\infty/K) \cdot n + \nu(K_\infty/K) 
\end{align} 
for the $p$-valuations of the class numbers of the intermediate fields $K_n$ contained in a $\Z_p$-extension $K_\infty/K$. Therefore the growth of these class numbers is described in terms of the \emph{Iwasawa invariants} $\mu(K_\infty/K)$, $\lambda(K_\infty/K)$ and $\nu(K_\infty/K)$. 

More recent Iwasawa theory has generalised on Iwasawa's approach in several directions. On the one hand, one studies more general and \emph{non-necessarily abelian} normal extensions of number fields, the Galois groups of which are nowadays usually assumed to be compact $p$-adic Lie groups. On the other hand, one studies the growth of \emph{Selmer groups} of abelian varieties or, even more generally, of certain Galois representations, over such extensions. 

Whereas the classical Iwasawa modules arising from the ideal class groups in a $\Z_p$-extension $K_\infty/K$ are always torsion modules over the corresponding Iwasawa algebra $\Z_p\llbracket \Gal(K_\infty/K)\rrbracket $ (see also \cite{iwasawa}), it turned out that the Selmer groups often yield something non-(co-)torsion (for example, this occurs naturally in the case of supersingular reduction at $p$). Therefore Greenberg, Coates--Sujatha, and others initiated the investigation of \emph{fine Selmer groups} which are canonical subgroups of the Selmer groups. The \emph{weak Leopoldt conjecture} predicts the corresponding Iwasawa modules (obtained by taking Pontryagin duals) to be torsion over the Iwasawa algebra. It is widely believed that this conjecture holds true in a very general setting. We will, as a side-product of our main results, prove that the conjecture holds under very natural hypotheses (see Corollary~\ref{cor:weak-Leo} and Theorem~\ref{thm:weak-Leo_dense}). 

Therefore both fine Selmer groups and ideal class groups (should) yield torsion Iwasawa modules. In their seminal paper \cite{coates-sujatha}, Coates and Sujatha showed that the analogy between these two objects goes much farther. For example, it has been conjectured since Iwasawa that the $\mu$-invariant of the \emph{cyclotomic} $\Z_p$-extension is zero for any number field $K$. Analogously, Coates and Sujatha conjectured that the $\mu$-invariant of the fine Selmer group of any elliptic curve over the cyclotomic $\Z_p$-extension should always vanish (cf. \cite[Conjecture~A]{coates-sujatha}). On the contrary, it is known that the $\mu$-invariant of the \emph{Selmer group} can be non-trivial even for the cyclotomic $\Z_p$-extensions of number fields (see \cite[Section~10]{mazur}). 

In this context, Lim and Murty have proved that the $\mu$-invariant of the cyclotomic $\Z_p$-extension $K_\infty^c$ of a number field $K$ vanishes if and only if for any abelian variety $A$ defined over $K$ such that $K$ contains the group $A[p]$ of $p$-torsion points on $A$, the $\mu$-invariant of the fine Selmer groups of $A$ over $K_\infty^c/K$ vanishes (see \cite[Theorem~5.5]{lim-murty}). One of the aims of this article is to generalise this statement to more general, so-called \emph{uniform} $p$-extensions (for the definitions we refer to Section~\ref{section:notation_extensions}). 

The main technical result (see Theorem~\ref{newrankthm}) will enable us, roughly speaking, to compare the $p^k$-torsion subgroups, $k \in \N$, of the fine Selmer group of an abelian variety $A$ defined over a number field $L$ on the one hand with the $p^k$-torsion subgroup of the ideal class group of $L$ on the other hand. This generalises the approach of Lim and Murty and works under the hypothesis that $L$ contains $A[p^k]$. 

Now let $K_\infty/K$ be a normal uniform $p$-extension (see Section~\ref{section:notation_extensions} for more details), and let $A$ be an abelian variety defined over $K$. Let $Y(K_\infty)$ denote the projective limit of the ideal class groups of the intermediate fields $K_n \subseteq K_\infty$ fixed by $\Gal(K_\infty/K_n)^{p^n}$, and let $Y_A^{(K_\infty)}$ be the projective limit of the Pontryagin duals of the fine Selmer groups of $A$ over the $K_n$ (for the definitions, we refer to Section~\ref{section:notation_fine_Selmer_Groups}). We want to compare the structural invariants of the  Iwasawa modules $Y(K_\infty)$ and $Y_{A}^{(K_\infty)}$. By \emph{structural invariants}, we mean the rank over the Iwasawa algebra $\Z_p\llbracket \Gal(K_\infty/K)\rrbracket $, the $\mu$-invariant and also finer invariants like $\Z_p\llbracket H\rrbracket $-ranks, ${H \subseteq G}$ (see Theorem~\ref{thm:B} below), and, in the special case of multiple $\Z_p$-extensions, the generalised $\lambda$-invariant $l_0$ (see Theorems~\ref{thm:C} and \ref{thm:D}). 

It turns out that the ranks and the $\mu$-invariants can be compared if $K_\infty$ contains $A[p^k]$ for a sufficiently large $k$. More precisely, we have the following result, for which we introduce some more notation. Let $G = \Gal(K_\infty/K)$. For any finitely generated $\Z_p\llbracket G\rrbracket $-module $X$, we denote by $r_{\Z_p\llbracket G\rrbracket }(X)$ the rank as a $\Z_p\llbracket G\rrbracket $-module, and we let 
$$ f_i(X) = \rg_{\F_p\llbracket G\rrbracket }(p^i X[p^\infty] / p^{i+1}X[p^\infty]), $$ 
$i \in \N$, where ${X[p^\infty] \subseteq X}$ denotes the subgroup of $p$-power torsion elements. Then $f_i(X) = 0$ for each $i \gg 0$, and the $\mu$-invariant of $X$ is defined as ${\mu(X) = \sum_{i \ge 0} f_i(X)}$ (for details, see Section~\ref{section:2}). 

\begin{thm} \label{thm:A} 
  Let $A$ be an abelian variety defined over the number field $K$, and let $K_\infty/K$ be a normal uniform pro-$p$ extension with Galois group $G$. We choose a finite set $\Sigma$ of primes of $K$ which consists of the primes above $p$, the primes where $A$ has bad reduction and the primes of $K$ which ramify in $K_\infty$. If $p = 2$, then we assume that $K$ is totally imaginary. 
  
  Suppose that no prime $v \in \Sigma$ splits completely in $K_\infty/K$, and that $A[p^k] \subseteq K_\infty$. 
  
  Then 
 $$v \cdot r_{\Z_p\llbracket G \rrbracket }(Y_{A}^{(K_\infty)}) + \sum_{i=0}^{v-1}f_i(Y_{A}^{(K_\infty)}) =  \sum_{i=0}^{v-1}f_i(Y(K_\infty))$$ 
 for every $v\le k$. 
\end{thm} 
In particular, if $k$ is larger than the maximal $i$ such that either $f_i(Y_{A}^{(K_\infty)})$ or $f_i(Y(K_\infty))$ is non-zero, then we actually obtain 
$$ k \cdot r_{\Z_p\llbracket G \rrbracket }(Y_{A}^{(K_\infty)}) + \mu(Y_{A}^{(K_\infty)}) = \mu(Y(K_\infty)). $$ 
Here we note that $r_{\Z_p\llbracket G \rrbracket }(Y(K_\infty)) = 0$ by the hypotheses from the theorem. 

We also prove a variant of Theorem~\ref{thm:A} which proves only an inequality instead of the above equality, but holds  under the weaker assumption that ${p^k \cdot A[p^\infty] \ne \{0\}}$ (see Theorem~\ref{thm:muungleichung}). 

If $K_\infty$ contains the full $p$-torsion $A[p^\infty]$, then we can compare also finer structural invariants generalising the classical $\lambda$-invariants.
\begin{thm} \label{thm:B} 
   Let $K_\infty/K$ and $\Sigma$ be as above, and let ${H \subseteq G = \Gal(K_\infty/K)}$ be a uniform subgroup. We  suppose that $Y_{A}^{(K_\infty)}$ is finitely generated over $\Z_p\llbracket H\rrbracket $, and that ${A[p^\infty] \subseteq K_\infty^H}$. 
   
   If all primes of $\Sigma$ are finitely split in $K_\infty^H/K$, then ${\rg_{\Z_p\llbracket H\rrbracket }(Y(K_\infty))}$ is also finite and 
   \[ \rg_{\Z_p\llbracket H\rrbracket }(Y_{A}^{(K_\infty)}) = 2d \cdot \rg_{\Z_p\llbracket H\rrbracket }(Y(K_\infty)). \] 
   Here $d$ denotes the dimension of the abelian variety $A$. 
\end{thm} 
We briefly describe the arithmetic meaning of the $\Z_p\llbracket H\rrbracket $-ranks. Suppose that the abelian variety $A$ has good ordinary reduction at the primes of $K$ above $p$, let $K_\infty^c$ be the cyclotomic $\Z_p$-extension of $K$, and let $K_n^c \subseteq K_\infty^c$ be the subfield of degree $p^n$ over $K$, $n \in \N$. Let $X_A^{(K_\infty^c)}$ be the Pontryagin dual of the (ordinary) Selmer group, and suppose that $X_A^{(K_\infty^c)}$ is a torsion $\Z_p\llbracket \Gal(K_\infty^c/K)\rrbracket $-module. It has been shown by Mazur (see \cite[p.~185]{mazur}) that we obtain bounds for Mordell-Weil ranks: 
\[ \rg_{\Z}(A(K_n^c)) \le \lambda(X_A^{(K_\infty^c)}) \] 
for every $n \in \N$. In \cite{Hung-Lim}, the authors considered, more generally, uniform $p$-adic Lie extensions $K_\infty/K$ of dimension $l \ge 2$ which contain $K_\infty^c$. They proved that 
\[ \rg_{\Z}(A(K_n)) \le \rg_{\Z_p\llbracket H\rrbracket }(X_A^{(K_\infty)}/X_A^{(K_\infty)}[p^\infty]) \cdot p^{(l-1)n} + \Ok(p^{(l-2)n}), \] 
provided that $X_A^{(K_\infty)}/X_A^{(K_\infty)}[p^\infty]$ is finitely generated over $\Z_p\llbracket H\rrbracket $. In this context, we also would like to mention a recent related result by A.~Ray (see \cite[Theorem~2.6]{ray}).  

Therefore the $\Z_p\llbracket H\rrbracket $-rank of the Pontryagin dual of the \emph{Selmer group} $X_{A}^{(K_\infty)}$, or more precisely, the rank of the quotient $X_{A}^{(K_\infty)}/X_{A}^{(K_\infty)}[p^\infty]$, can (if it is finite) be seen as a higher-dimensional analogue of the classical Iwasawa $\lambda$-invariant (see also \cite[Section~6]{coates-howson}). In this article, we consider \emph{fine Selmer groups}. By a classical conjecture which goes back to Coates and Sujatha (see \cite[Conjecture~A]{coates-sujatha}), the $\mu$-invariant of the fine Selmer group over the cyclotomic $\Z_p$-extension $K_\infty^c$ of $K$ should be zero. Building on this conjecture, the $\mu$-invariants are expected to be zero for fine Selmer groups over uniform extensions $K_\infty$ which contain $K_\infty^c$, and one can thus study the $\Z_p\llbracket H\rrbracket $-rank of $Y_A^{(K_\infty)}$ itself instead of $Y_A^{(K_\infty)}/Y_A^{(K_\infty)}[p^\infty]$. It seems not known whether the $\lambda$-invariants (or, more generally, $\Z_p\llbracket H\rrbracket $-ranks) of the corresponding Pontryagin duals still serve as upper bounds for the Mordell-Weil ranks. On the other hand, whereas the finiteness of ${\rg_{\Z_p\llbracket H\rrbracket }(X_A^{(K_\infty)}/X_A^{(K_\infty)}[p^\infty])}$ is known only in very few cases (see \cite{MHG}), the finiteness of ${\rg_{\Z_p\llbracket H\rrbracket }(Y_A^{(K_\infty)})}$ for suitable ${H \subseteq G}$ holds under very natural assumptions (see Theorem~\ref{thm:Conj_A}).

For multiple $\Z_p$-extensions $K_\infty/K$ of number fields, one can define another natural generalised $\lambda$-invariant. Cuoco and Monsky have proved an asymptotic formula for the growth of the $p$-valuation $e_n$ of the class numbers of the intermediate number fields which generalises Iwasawa's formula~\eqref{eq:iw-intro}: if $K_\infty/K$ denotes a $\Z_p^l$-extension with intermediate fields $K_n$, $n \in \N$, then there exist integers $m_0(K_\infty/K), l_0(K_\infty/K)$ such that 
\[ e_n = m_0(K_\infty/K) p^{ln} + l_0(K_\infty/K) n p^{(l-1)n} + \Ok(p^{(l-1)n}). \] 
In particular, if $l = 1$, then ${m_0(K_\infty/K) = \mu(K_\infty/K)}$ and ${l_0(K_\infty/K) = \lambda(K_\infty/K)}$. In fact, these \emph{generalised Iwasawa invariants} ${m_0, l_0 \in \N}$ can be defined for any finitely generated $\Z_p\llbracket G\rrbracket $-module, provided that ${G \cong \Z_p^l}$ (see Section~\ref{section:notation_extensions}). This enables us to define also a $l_0$-invariant for fine Selmer groups. In our third comparison result (see Theorem~\ref{thm:l0_Gleichheit} below), we compare ${l_0(Y(K_\infty)) = l_0(K_\infty/K)}$ and $l_0(Y_{A}^{(K_\infty)})$ in the following setting. 
\begin{thm} \label{thm:C} 
   Let $A$ be an abelian variety of dimension $d$ defined over $K$, and let $K_\infty/K$ be a $\Z_p^l$-extension, $p \ge 3$, for some $l \ge 2$. We let $\Sigma_p$ denote the set of primes of $K$ above $p$, and we suppose that the decomposition subgroup ${D_v \subseteq \Gal(K_\infty/K)}$ has dimension at least 2 for each ${v \in \Sigma_p}$. 
   
   If $A[p^\infty] \subseteq K_\infty$, then 
   \[ l_0(Y_{A}^{(K_\infty)}) = 2d \cdot l_0(Y(K_\infty)). \]  
\end{thm} 
In order to prove this result, we first derive a $p^n$-quotient version of the class number formula of Cuoco and Monsky (see Theorem~\ref{thm:ckmm}): 
\[ v_p(|Y(K_n)/p^n Y(K_n)|) = m_0(K_\infty/K) p^{ln} + l_0(K_\infty/K) n p^{(l-1)n} + \Ok(p^{(l-2)n}). \] 
Here $Y(K_n)$ denotes the $p$-primary subgroup of the ideal class group of $K_n$, $n \in \N$, and $v_p$ is the $p$-adic valuation, normalised so that $v_p(p) = 1$. 

Then we apply our comparison result (Theorem~\ref{newrankthm}) in order to derive a similar asymptotic formula for the growth of $v_p(|Y_{A}^{(K_n)}/p^n Y_{A}^{(K_n)}|)$. Finally, we use a control theorem (see Lemma~\ref{lemma:katharinas_control_thm}) in order to deduce a formula for 
$$v_p(|(Y_{A}^{(K_\infty)})_{G_n}/p^n (Y_{A}^{(K_\infty)})_{G_n}|), $$ 
where ${G_n = \Gal(K_\infty/K_n)}$ for every ${n \in \N}$. But the growth of these quotients of Galois coinvariants can also be described in terms of the generalised Iwasawa invariants of $Y_{A}^{(K_\infty)}$ -- this provides a link between $l_0(Y(K_\infty))$ and $l_0(Y_{A}^{(K_\infty)})$. 

Finally, we compare $l_0(Y_{A}^{(K_\infty)})$ and $\rg_{\Z_p\llbracket H\rrbracket }(Y_{A}^{(K_\infty)})$ in a suitable setting (see Theorem~\ref{thm:l0absch}): 
\begin{thm} \label{thm:D} 
  Let $K_\infty/K$ and $A$ be as in Theorem~\ref{thm:C}. Let $K_\infty^{(1)}$ be a $\Z_p$-extension of $K$ contained in $K_\infty$, and write ${H = \Gal(K_\infty/K_\infty^{(1)})}$. If $Y_{A}^{(K_\infty)}$ is finitely generated over $\Z_p\llbracket H\rrbracket $ and the $\mu$-invariant of $Y_{A}^{(K_\infty)}$ as a $\Z_p\llbracket H\rrbracket $-module is zero, then 
  \[ \rg_{\Z_p\llbracket H\rrbracket }(Y_{A}^{(K_\infty)}) \ge 2d l_0(Y(K_\infty)) = l_0(Y_{A}^{(K_\infty)}). \] 
\end{thm} 

The paper is organised as follows. After introducing the necessary background on uniform extensions, Iwasawa modules and fine Selmer groups in Section~\ref{section:2}, we recall several facts on abelian varieties with complex multiplication that will provide us with a class of interesting examples. In Section~\ref{subsection:classnumberformulas}, we collect known results on asymptotic class number formulas in $p$-adic Lie extensions. 

The technical auxiliary results needed for the proofs of our main results are proved in Section~\ref{section:5} (growth of Galois coinvariant modules of Iwasawa modules, new asymptotic class number formulas, and auxiliary results on $\mu$-invariants) and in Section~\ref{section:6} (control theorems). 

In Section~\ref{section:main}, we prove our main results, including Theorems~\ref{thm:A}, \ref{thm:B}, \ref{thm:C} and \ref{thm:D}, and also several interesting applications and corollaries. 

\vspace{2mm} 
\textbf{Acknowledgements.} We thank Meng Fai Lim and Debanjana Kundu for kindly sending us a preprint of their work and answering our questions. We would also like to thank Damaris Schindler for supporting a visit of the first author in G\"ottingen. Part of this research was conducted while the second author was a postdoc at the University of G\"ottingen.

\section{Background and notation} \label{section:2} 
\subsection{General notation} \label{section:notation_general} 
We denote by $\N$ the set of natural numbers, which includes 0. Let ${p \in \N}$ be a prime number, which will be fixed throughout the article. For any finitely generated group $G$, we denote by $G[p^\infty]$ the subset of elements of $p$-power torsion elements. If $G$ is abelian, then $G[p^\infty]$ is a subgroup of $G$. 

For any discrete $\Z_p$-module $M$, we denote by 
\[ M^\vee = \textup{Hom}_{\textup{cont}}(M, \Q_p/\Z_p)\] 
the \emph{Pontryagin dual} of $M$ (here $\textup{Hom}_{\textup{cont}}$ means the set of continuous homomorphisms). 

In what follows, we fix a number field $K$. If $N/L$ are any algebraic (possibly infinite) extensions of $K$  and $v$ is a prime of $L$, then we denote by $\Sigma_v(N)$ the set of primes of $N$ above $v$. Most prominently, the set $\Sigma_p := \Sigma_p(K)$ of primes above $p$ in our base field $K$ will play an important role throughout the paper. More generally, for any finite set $\Sigma$ of primes of $K$ and $N$ as above, we denote by $\Sigma(N)$ the set of primes of $N$ which divide some $v \in \Sigma$. 

If $A$ is an abelian variety defined over the number field $K$, then we will denote by $\Sigma_{\textup{br}}(A)$ the finite set of primes of $K$ where $A$ has bad reduction. 

\subsection{Uniform extensions} \label{section:notation_extensions} 
Let $G$ be a finitely generated pro-$p$ group. Let $(G_n)_{n \ge 0}$ be the subgroups of the lower $p$-series of $G$ (see \cite[Definition~1.15]{Dixon}; in the notation of \cite{Dixon}, we let $G_n = P_{n+1}(G)$ for every $n \in \N$). 

\begin{def1}
$G$ is called \emph{uniform of dimension $l \in \N$} if $G$ is profinite, topologically generated by $l$ generators, and there exists a filtration of $G$ by the $G_n$, i.e. we have 
\[ G = G_0 \supseteq G_1 \supseteq \ldots \supseteq G_n \supseteq \ldots \] 
such that each $G_{n+1}$ is normal in $G_n$ and ${G_n/G_{n+1} \cong (\Z/p\Z)^l}$. 
\end{def1}
If $K_\infty/K$ is a normal $p$-adic Lie extension with Galois group $G$ being uniform of dimension $l$, then there exists a canonical sequence of intermediate fields $K_n \subseteq K_\infty$ of degree $p^{ln}$ over $K$, fixed by the subgroups $G_n := P_{n+1}(G)$. Since each $G_n \subseteq G$ is a normal subgroup in view of \cite[Lemma~2.4]{Dixon}, the $K_n/K$ are Galois extensions. 

In view of \cite[Theorem~3.6]{Dixon}, we have $G_n = G^{p^n}$ for every $n \in \N$. Note: although $G$ need not be commutative, the subset of $p^n$-th powers of elements of $G$ automatically forms a subgroup because $G$ is uniform. Moreover, each uniform pro-$p$-group $G$ is torsion-free, i.e. it satisfies $G[p^\infty] = \{0\}$ (see \cite[Theorem~4.5]{Dixon}). 

In this article, we consider only normal uniform extensions $K_\infty/K$. Since any compact $p$-adic Lie group contains an open normal uniform subgroup by a theorem of Lazard (see \cite[Corollary~8.34]{Dixon}), this is not a severe restriction. The index $n$ will always mark the subgroup $G_n = P_{n+1}(G) = G^{p^n}$ of $G$ fixing the appropriate intermediate field $K_n \subseteq K_\infty$, $n \in \N$. 

If $G$ is a uniform pro-$p$-group, then the completed group ring is a Noetherian ring without zero divisors by \cite[Theorem~2.3]{coates-howson}. For any finitely generated $\Z_p\llbracket G\rrbracket $-module $X$, we define its $\Z_p\llbracket G\rrbracket $-rank by 
$$ \rg_{\Z_p\llbracket G\rrbracket }(X)= \dim_{\mathcal{Q}(G)}(\mathcal{Q}(G) \otimes_{\Z_p\llbracket G\rrbracket }(X)), $$ 
where $\mathcal{Q}(G)$ denotes the skew field of fractions of $\Z_p\llbracket G\rrbracket $ (see \cite[Chapter~10]{goodearl-warfield}). Following Howson (see \cite[(33)]{howson}), we define the $\mu$-invariant of such a finitely generated $\Z_p\llbracket G\rrbracket $-module $X$ as 
\begin{align} \label{eq:mu}  \mu_{\Z_p\llbracket G\rrbracket }(X) = \sum_{i \ge 0} \rg_{\F_p\llbracket G\rrbracket }(p^i X[p^\infty]/p^{i+1} X[p^\infty]), \end{align} 
where $\F_p = \Z_p/p \Z_p$ is the finite field with $p$ elements. 

Throughout the article, if the group ring is clear from the context, then we simply write $\mu(X)$. 

The most classical example of uniform pro-$p$-groups are the (commutative) direct products $\Z_p^l$, $l \ge 1$. A \emph{$\Z_p^l$-extension} $K_\infty/K$ is a normal extension such that ${G = \Gal(K_\infty/K)}$ is isomorphic to $\Z_p^l$. In this case, the theory of finitely generated $\Z_p\llbracket G\rrbracket $-modules is particularly well understood. If $G \cong \Z_p^l$, then $\Z_p\llbracket G\rrbracket $ is isomorphic to the ring $\Lambda_l = \Z_p\llbracket T_1, \ldots, T_l\rrbracket $ of formal power series with coefficients in $\Z_p$. This local ring is a unique factorisation domain. By a general structure result, each finitely generated $\Z_p\llbracket G\rrbracket $-module $X$ is pseudo-isomorphic to an elementary $\Lambda_l$-module of the form 
\begin{align} \label{eq:elementary}  E_X = \Lambda_l^r \oplus \bigoplus_{i = 1}^s \Lambda_l/(p^{e_i}) \oplus \bigoplus_{j = 1}^t \Lambda_l/(h_j), \end{align}  
where $h_1, \ldots, h_t \in \Lambda_l$ are coprime with $p$ and where \emph{pseudo-isomorphic} means that there exists a $\Lambda_l$-module homomorphism $\varphi: X \longrightarrow E_X$ whose kernel and cokernel are \emph{pseudo-null}, i.e. annihilated by two relatively prime elements of the UFD $\Lambda_l$. The \emph{characteristic power series} $F_X$ of $X$ is defined as $F_X = p^{\sum_{i=1}^s e_i} \cdot \prod_{j=1}^t h_j$. 

If $l = 1$, then the $h_j$ are in fact so-called distinguished polynomials. One defines the \emph{Iwasawa invariants} of such a $\Lambda_1$-module $X$ as ${\mu(X) = \sum_{i = 1}^s e_i}$ and ${\lambda(X) = \sum_{j = 1}^t \deg(h_j)}$, and one can show that this definition of the $\mu$-invariant coincides with the definition~\eqref{eq:mu} given above (see \cite[Lemma~2.4]{KM21}). For arbitrary $l$, Cuoco and Monsky defined in \cite{cuoco-monsky} the \emph{generalised Iwasawa invariants} of a finitely generated $\Lambda_l$-module $X$ via identification of $\Lambda_l$ with $\Z_p\llbracket G\rrbracket $ (recall that $G \cong \Z_p^l$). Fixing an elementary $\Lambda_l$-module $E_X$ as in~\eqref{eq:elementary} above, the generalised $\mu$-invariant of $X$ is defined as $m_0(X) = \sum_{i=1}^s e_i$. Again, one can show that $m_0(X) = \mu(X)$ in the sense of \eqref{eq:mu}. 

The definition of a generalised Iwasawa $\lambda$-invariant is more involved. Write $F_X = p^{m_0(X)} \cdot f_X$ for some power series $f_X \in \Lambda_l$ which is coprime with $p$, and let $\overline{f_X} \in \Lambda_l/p \Lambda_l$ be the reduction modulo $p$. Then $l_0(X) = \sum_{\mathcal{P}} v_{\mathcal{P}}(\overline{f_X})$, where the sum runs over all primes of $\Lambda_l/p \Lambda_l$ of the form $\mathcal{P} = (\overline{g-1})$ for some $g \in G \setminus G^p$. Here $v_{\mathcal{P}}$ denotes the $\mathcal{P}$-adic valuation. 

In Section~\ref{subsection:classnumberformulas}, we define the (generalised) Iwasawa invariants of a (multiple) $\Z_p$-extension, and we will describe their arithmetical meaning. Note that for each number field $K$, there exists at least one $\Z_p$-extension of $K$, the so-called \emph{cyclotomic $\Z_p$-extension}. In this article, we will denote this cyclotomic $\Z_p$-extension of $K$ by $K_\infty^c$.

\subsection{Fine Selmer groups} \label{section:notation_fine_Selmer_Groups} 
For a number field $K$ and a prime $v$ of $K$, we denote by $K_v$ the completion of $K$ at $v$. We fix an algebraic closure $\overline{K}$ of $K$, and we denote by ${G_K = \Gal(\overline{K}/K)}$ the absolute Galois group of $K$. For every $G_K$-module $X$, we let ${H^i(K,X) = H^i(G_K,X)}$ be the corresponding Galois cohomology groups, $i \in \N$. Moreover, for every relative Galois extension $L/K$ and every $\Gal(L/K)$-module $X$, we write ${H^i(L/K, X) = H^i(\Gal(L/K),X)}$. 

Now let $A$ be an abelian variety defined over the number field $K$, let $\Sigma$ be a finite set of primes of $K$ which contains all the primes above $p$ and the primes where $A$ has bad reduction, and fix $K$, $A$ and $\Sigma$. Let $K_\Sigma$ be the maximal algebraic (non-necessarily abelian) extension of $K$ which is unramified outside of $\Sigma$. \begin{def1}
 In analogy with \cite{coates-sujatha} and \cite{lim-murty}  we define the $p$-primary \emph{fine $\Sigma$-Selmer group} of $A$ over $K$ as 
\[ \Sel_{0,A, \Sigma}(K) = \ker \left( H^1(K_\Sigma/K,A[p^\infty]) \longrightarrow \prod_{v \in \Sigma} H^1(K_v,A[p^\infty])\right).\] \end{def1}
Note that by \cite[Lemma 4.1]{lim-murty} this definition does not depend on the finite set $\Sigma$, and we can therefore omit $\Sigma$ from the notation.
If $L$ is a number field such that $K \subseteq L \subseteq K_\Sigma$, then we can define analogously the fine Selmer group $\Sel_{0,A}(L)$. If $M \subseteq K_\Sigma$ is an infinite extension of $K$, then we define 
\[ \textup{Sel}_{0,A}(M)=\varinjlim_{K \subseteq L \subseteq M} \textup{Sel}_{0,A}(L), \] 
where $L \subseteq M$ runs through all the finite subextensions. Note that $L_\Sigma = K_\Sigma$ for each such $L$, since $M/K$ is unramified outside of $\Sigma$, and therefore $\Sel_{0,A}(L)$ is a subgroup of $H^1(K_\Sigma/L,A[p^\infty])$ for each such $L$. 

We also need a $p^k$-version of the fine Selmer groups.  
\begin{def1}
For $k \in \N$ we define the  
\emph{$p^k$-fine Selmer groups} by
\[ \Sel_{0,A[p^k], \Sigma}(L) = \ker \left( H^1(K_\Sigma/L, A[p^k]) \longrightarrow \prod_{w \in \Sigma(L)} H^i(L_w, A[p^k])\right) \] 
for every number field $L \subseteq K_\Sigma$ containing $K$.\end{def1} By taking limits we can define these groups for arbitrary algebraic extensions $M \subseteq K_\Sigma$ of $K$. Note that the latter definition really depends on $\Sigma$. So in this case it is essential to keep the superscript $\Sigma$ (see \cite[proof of Theorem~5.1]{lim-murty}).

Now suppose that $K_\infty/K$ is a uniform $p$-adic Lie extension, and let $\Sigma$ be a finite set of primes of $K$ which contains $\Sigma_p$, $\Sigma_{\textup{br}}(A)$ and all the primes of $K$ which ramify in $K_\infty$. Then we can define fine Selmer groups of $A$ over each intermediate field $K_n \subseteq K_\infty$. We denote the corresponding Pontryagin duals by 
\[ Y_{A}^{(L)} = \Sel_{0,A}(L)^\vee, \] 
and we define the projective limit 
\[ Y_{A}^{(K_\infty)} = \varprojlim_n Y_{A}^{(K_n)}\] 
with respect to the corestriction maps.

\section{Abelian varieties with complex multiplication} \label{section:3} 
In this section we summarise the facts needed on abelian varieties with complex multiplication. We define this notion following the seminal article \cite{serre-tate} of Serre and Tate. 
\begin{def1} 
  Let $K$ be a number field. An abelian variety $A$ of dimension $d$ defined over $K$ \emph{has complex multiplication by a CM number field $F$} if $[F:\Q] = 2d$ and there exists an injection $F \hookrightarrow \Q \otimes_\Z \textup{End}_K(A)$.
\end{def1} 

\begin{thm}[Serre-Tate] \label{thm:serre-tate} 
  Let $A$ be an abelian variety defined over the number field $K$, and suppose that $A$ has complex multiplication. Then $A$ has potentially good reduction everywhere. In fact, there exists a finite cyclic extension $K'/K$ such that $A$ has good reduction at every prime of $K'$. 
\end{thm} 
\begin{proof} 
  See \cite[Theorem~7]{serre-tate}. 
\end{proof} 

In our applications, we are interested in studying uniform $p$-extensions $K_\infty/K$ such that $A[p^n] \subseteq A(K_n)$ for every $n \in \N$. In order to show that such extensions do exist for abelian varieties with complex multiplication, we make use of the following 
\begin{thm} \label{thm:CM-suff} 
  Let $K$ be a number field, and let $A$ be an abelian variety defined over $K$ which has complex multiplication. Then there exist a finite extension $K'$ of $K$ and a $\Z_p^l$-extension of $K'$ which contains $A[p^\infty]$. In fact, the extension $K'(A[p^\infty])/K'$ is \emph{abelian} and unramified outside of $p$. 
\end{thm} 
\begin{proof} 
   It follows from \cite[Chapter~4, Theorem~1.1]{lang_CM} that $K(A[p^\infty])/K$ is abelian. If $K'$ is large enough such that $A$ has good reduction at each prime of $K'$, then the shifted extension $K'(A[p^\infty])/K'$ is unramified outside of $p$ by \cite[Corollary~2(b)]{serre-tate}, and the assertion of the theorem follows from \cite[Corollary~13.6]{wash} (it might be necessary to enlarge $K'$ again if the class number of $K'$ is greater than 1). 
\end{proof} 

Combining the last two theorems, we may derive the following 
\begin{cor} \label{cor:pntorsion} 
  Let $K$ be a number field, and let $A$ be an abelian variety with complex multiplication defined over $K$. Then there exist a finite extension $K'$ of $K$ and a $\Z_p^l$-extension $K_\infty'$ of $K'$ such that $A[p^n] \subseteq A(K'_{n+c})$ for every $n \in \N$, where $c$ is a fixed constant. 
\end{cor} 
\begin{proof} 
  Suppose that $K'$ has been chosen such that $A$ has good reduction everywhere and such that $K'(A[p^\infty])/K'$ is a $\Z_p^l$-extension (in particular, we assume this extension to be pro-$p$). 
  
  Let $\sigma \in \Gal(K'(A[p^{n+1}])/K'(A[p^n]))$. We show that $\sigma^p =\id$. Choose any $Q \in A[p^{n+1}]\setminus A[p^n]$. Then the order of $Q$ is equal to $p^{n+1}$. Since 
  $$\sigma(Q)^p = \sigma(Q^p) = Q^p$$ 
  because $\sigma$ fixes $K'(A[p^n])$, $\sigma(Q) - Q \in A[p]$. Writing $\sigma(Q) = Q + R$ for some ${R \in A[p]}$, we may conclude that ${\sigma^i(Q) = Q + R^i}$ for every $i \in \N$. In particular, ${\sigma^p(Q) = Q}$. Since this holds for every ${Q \in A[p^{n+1}]}$, we may conclude that $\sigma^p = \id$. 
  
  This proves that the extension $K'(A[p^{n+1}])/K'(A[p^n])$ is $p$-elementary for every $n \in \N$. The corollary then follows by induction; in fact, the integer $c$ can be chosen as the smallest index $r$ such that $A[p] \subseteq A(K_r')$. 
\end{proof} 

We conclude this section by mentioning an auxiliary lemma which will be used below and which in fact holds for any abelian variety over $K$. 

\begin{lemma}
\label{growth-condition}
Let $K$ be any field of characteristic zero and let $A$ be an abelian variety defined over $K$. Let $K_\infty/K$ be a uniform pro-$p$-extension of $K$ of dimension $l$, with intermediate fields $K_n$. Then \[v_p(|A(K_n)[p^\infty]|)=\mathcal{O}(n).\] 
\end{lemma}
Note that the bound $\mathcal{O}(n)$ is far from being optimal. For example, if we consider $K=\Q$ and $K_\infty=\Q_\infty^{\textup{c}}$, then it is well-known that for every elliptic curve $E$, the group $E(\Q^{\textup{c}}_\infty)[p^\infty]$ is finite. So in this case Lemma~\ref{growth-condition} has to be understood as a (coarse) upper bound.
\begin{proof}
   Fix $n$ and let $p^N$ be the exponent of $A(K_n)[p^\infty]$. Assume that there is a $p^{N+1}$-torsion point $Q$ in $K_{\infty}$. It is straight-forward to see that $[K(Q):K]\ge p$. As $[K_{n+1}:K_n]=p^l$, we see that the exponent of $A(K_m)[p^\infty]$ is bounded by $p^{N+(m-n)l}$ for all $m\ge n$. So we can conclude that $A(K_m)[p^\infty]\subseteq A[p^{N+(m-n)l}]$ which has $p$-valuation $2d(N+(m-n)l)$, where $d = \dim(A)$. The claim follows.
\end{proof}

\section{Asymptotic class number formulas} \label{subsection:classnumberformulas} 
For any number field $L$, we denote by $\textup{Cl}(L)$ the ideal class group of $L$, and we let $Y(L) \subseteq \textup{Cl}(L)$ denote its $p$-Sylow subgroup. If $M$ is any (possibly infinite) algebraic extension of the number field $K$, then we define 
\[ \textup{Cl}(M) = \varprojlim_{K \subseteq L \subseteq M} \textup{Cl}(L) \quad \text{ and } \quad Y(M) = \varprojlim_{K \subseteq L \subseteq M} Y(L). \] 
Here $L$ runs over the finite intermediate extensions, and the projective limits are taken with respect to the norm maps. 

Moreover, for a number field $L$ and any finite set $\Sigma$ of primes of $L$, we let $\textup{Cl}_\Sigma(L)$ and $Y_\Sigma(L)$ denote the quotients of $\textup{Cl}(L)$ and $Y(L)$ by the subgroups which are generated by the primes of $\Sigma$. 
For a possibly infinite algebraic extension $M$ of $K$ and a finite set $\Sigma$ of primes of $K$, the projective limits $\textup{Cl}_\Sigma(M)$ and $Y_\Sigma(M)$ are defined in the natural way (for each finite subextension $L$ of $M$, we divide out the subgroup generated by the primes of $L$ which divide some $v \in \Sigma$). By class field theory, $Y_\Sigma(M)$ is isomorphic to $\Gal(H_\Sigma(M)/M)$, where $H_\Sigma(M)$ denotes the maximal abelian pro-$p$-extension of $M$ which is unramified and in which the primes of $M$ above $\Sigma$ are totally split. If $\Sigma = \emptyset$, then we simply write $H(M)$ for the maximal abelian unramified pro-$p$-extension of $M$. 

In the following, we give an overview of known asymptotic class number formulas on the growth of class numbers in $p$-adic Lie extensions of number fields. 
\subsection{$\Z_p$-extensions} 
Historically seen, the first example for a class of $p$-adic Lie extensions with an asymptotic class number formula is of course Iwasawa's theorem on the growth of class numbers in a $\Z_p$-extension $K_\infty/K$, with intermediate fields $K_n$ of degree $p^n$ over $K$ (see \cite{iwasawa}): for each sufficiently large $n \in \N$, we have 
\begin{align} \label{eq:iwasawa} v_p(|\textup{Cl}(K_n)|) = \mu(K_\infty/K) p^n + \lambda(K_\infty/K) n + \nu(K_\infty/K). \end{align} 
Here $\mu(K_\infty/K) = \mu(Y(K_\infty))$ and $\lambda(K_\infty/K) = \lambda(Y(K_\infty))$ are the Iwasawa invariants of the $\Lambda$-module ${Y(K_\infty) = \varprojlim_n Y(K_n)}$, $\Lambda = \Z_p\llbracket T\rrbracket $, as defined in the previous section. 

\subsection{Multiple $\Z_p$-extensions} 
More generally, if $K_\infty/K$ is a $\Z_p^l$-extension, $l \ge 1$, then we have the following asymptotic formula due to Cuoco and Monsky (see \cite{cuoco-monsky}): 
\begin{align} \label{eq:cuoco-monsky} v_p(|\textup{Cl}(K_n)|) = m_0(Y(K_\infty)) p^{ln} + l_0(Y(K_\infty)) n p^{(l-1)n} + \Ok(p^{(l-1)n}), \end{align}  
where $m_0(Y(K_\infty)), l_0(Y(K_\infty)) \in \N$ are the so-called \emph{generalised Iwasawa invariants} of the $\Z_p\llbracket \Gal(K_\infty/K)\rrbracket $-module $Y(K_\infty)$, as defined in the previous section.  

Later Monsky refined the above formula and proved that 
\[ v_p(|\textup{Cl}(K_n)|) = m_0(Y(K_\infty)) p^{ln} + l_0(Y(K_\infty)) n p^{(l-1)n} + \alpha p^{(l-1)n} + \Ok(n p^{(l-2)n}), \] 
where, however, $\alpha \in \R$ does not need to be rational. 

In \cite{local_max}, the first author has studied $\Z_p^2$-extensions $K_\infty/K$ with the easiest possible ramification behaviour: exactly one prime of $K$ ramifies in $K_\infty$, and it is totally ramified. In this case, we have 
\[ v_p(|\textup{Cl}(K_n)|) = m_0(Y(K_\infty)) p^{2n} + l_0(Y(K_\infty)) n p^n + a p^n + b + c + C_n \] 
for suitable integers $a,b,c$ and every sufficiently large $n$, where $(C_n)_{n \in \N}$ is a sequence of integers attached to the maximal pseudo-null $\Z_p\llbracket \Gal(K_\infty/K)\rrbracket $-\-sub\-mod\-ule $Y(K_\infty)^\circ$ of $Y(K_\infty)$. Moreover, if $Y(K_\infty)^\circ$ is annihilated by an element $\gamma - 1$, $\gamma \in \Gal(K_\infty/K) \setminus \{1\}$, then the $C_n$ may be absorbed into the formula, i.e. then there exist integers $a',b',c'$ such that 
\[ v_p(|\textup{Cl}(K_n)|) = m_0(Y(K_\infty)) p^{2n} + l_0(Y(K_\infty)) n p^n + a' p^n + b' + c' \]
for each sufficiently large $n$ (see \cite[Theorem~5.4]{local_max}). 

\subsection{Exact formulas in the non-commutative setting} 
Over the last years, also several results for non-commutative $p$-adic Lie extensions of number fields have been obtained; we mention just some of them. 

In \cite{lei_classnumbers}, Lei considered certain $\Z_p^{l-1} \rtimes \Z_p$-extensions $K_\infty/K$, $p \ne 2$. In this situation, write ${G = \Gal(K_\infty/K) \cong H \rtimes \Gamma}$, and let $K_{m,n} = K_\infty^{H_m \rtimes \Gamma_n}$, where ${H_m = H^{p^m}}$ and ${\Gamma_n = \Gamma^{p^n}}$, $m, n \in \N$. Under a strict ramification hypothesis (in particular, $K$ shall contain exactly one prime above $p$, and this prime shall be totally ramified in $K_\infty$), Lei proved that for a fixed $n \in \N$, there exist integers $\mu_n, \lambda_n$ such that 
\[ v_p(|\textup{Cl}(K_{m,n})|) = \mu_n p^{(l-1)m} + \lambda_n m p^{(l-2)m} + \Ok(p^{(l-2)m}). \] 
Moreover, if $Y(K_\infty)$ is finitely generated over $\Z_p\llbracket H\rrbracket $ and $l = 2$, then 
\begin{align} \label{eq:lei1} 
   v_p(|\textup{Cl}(K_{n,n})|) = \rg_{\Z_p\llbracket H\rrbracket }(Y(K_\infty))n p^n + \Ok(p^n). 
\end{align} 
In joint work with Liang, Lim generalised Lei's estimate~\eqref{eq:lei1} to ${\Z_p^{l-1} \rtimes \Z_p}$-extensions for arbitrary $l \ge 2$, under the same strict ramification hypotheses (see \cite{Lim-Liang}). 

Finally, under the hypothesis that ${Y_f(K_\infty)) := Y(K_\infty)/Y(K_\infty)[p^\infty]}$ is finitely generated over $\Z_p\llbracket H\rrbracket $, and $l = 2$, Lei obtained upper bounds 
\begin{align} \label{eq:lei2} 
  \tilde{r}^n_p(\textup{Cl}(K_{n,n})) = \mu(Y(K_\infty))p^{2n} + \rg_{\Z_p\llbracket H\rrbracket }(Y_f(K_\infty)) n p^n + \Ok(p^n). 
\end{align} 
Here the notation $\tilde{r}^n_p$ means the following: for any finitely generated abelian group $N'$ we define $\tilde{r}^n_p(N')$ by  \[|N'/p^n N'| = p^{\tilde{r}(N')}.\] 

It is not uncommon that the above formula is for the cardinality of the quotient $\textup{Cl}(K_{n,n})/p^n \textup{Cl}(K_{n,n})$ and not for $\textup{Cl}(K_{n,n})$ itself. 
In the last part of this section, we mention more such results. 

\subsection{Estimates for quotients of ideal class groups} 
The first such result seems to be a result due to Perbet (see \cite[Théorème~0.1]{perbet}, and cf. also Theorem~\ref{thm:perbet+} below): let $K_\infty/K$ be any uniform pro-$p$-extension of dimension $l$ with Galois group $G$ and intermediate fields $K_n = K_\infty^{G_n}$, $n \in \N$. Then 
\[ \tilde{r}^n_p(\textup{Cl}(K_{n,n})) = (n \cdot \rg_{\Z_p\llbracket G\rrbracket }(Y(K_\infty)) + \mu(Y(K_\infty)) p^{ln} + \Ok(n p^{(l-1)n}). \] 
Actually Perbet's theorem is formulated more generally for $S$-$T$-class groups. 

Building on the work of Perbet, Lim proved a generalisation of the estimates~\eqref{eq:lei2} of Lei (see \cite{Lim-classnumbers}): let $K_\infty/K$ be a uniform pro-$p$-extension of dimension $l$, $p \ne 2$, and suppose that $G = \Gal(K_\infty/K)$ contains a closed normal subgroup $H$ such that $G/H \cong \Z_p$. If $Y_f(K_\infty)$ (defined as above) is finitely generated over $\Z_p\llbracket H\rrbracket $, then 
\[ \tilde{r}^n_p(\textup{Cl}(K_n))\le \mu(Y(K_\infty)) p^{ln} + \rg_{\Z_p\llbracket H\rrbracket }(Y_f(K_\infty))n p^{(l-1)n} + \Ok(p^{(l-1)n}). \] 
We note that $\rg_{\Z_p\llbracket G\rrbracket }(Y(K_\infty)) = 0$ in this situation, since $Y_f(K_\infty)$ is finitely generated over $\Z_p\llbracket H\rrbracket $. 

Under the stronger assumption that $Y(K_\infty)$ itself is finitely generated over $\Z_p\llbracket H\rrbracket $, Lim proved in \cite{Lim-classnumbers} that 
\[ \tilde{r}^n_p(\textup{Cl}(K_n)) \le \rg_{\Z_p\llbracket H\rrbracket }(Y(K_\infty)) n p^{(l-1)n} + \mu_{\Z_p\llbracket H\rrbracket }(Y(K_\infty))p^{(l-1)n} + \Ok(n p^{(l-2)n}). \] 
We remark that Lim did not need any extra ramification hypothesis on $K_\infty/K$ in order to prove the above results.

\section{Auxiliary results} \label{section:5} 
\subsection{Ranks of groups} 
Let $k$ be a positive integer and $N$ a $\Z_p$-module. We define the $p^k$-rank of $N$, denoted by $r_p^k(N)$, as the natural number $r$ such that $\vert N[p^k]\vert =p^r$ -- whenever this quantity is finite. If $k=1$ this coincides with the well-known notion of $p$-ranks. For a finitely generated abelian group $N'$ we also recall from the previous section the definition of $\tilde{r}_p^k(N')$: 
\[ p^{\tilde{r}_p^k(N')}=\vert N'/p^k N'\vert. \] 
If $N$ is finite, then $r_p^k(N)=\tilde{r}_p^k(N)$. If $N$ is cofinitely generated, then $$ r_p^k(N)=\tilde{r}_p^k(N^\vee). $$  
\begin{lemma} 
\label{lem:ranksunderepimorphism}Let $Z$ and $Z'$ be cofinitely generated abelian groups.
If $Z'\subseteq Z$, then $r_p^k(Z')\le r_p^k(Z)$. If $Z\longrightarrow Z'$ is an epimorphism, then $r_p^k(Z')\le r_p^k(Z)$.
\end{lemma}
\begin{proof}
The first claim is obvious from the inclusion $Z'[p^k]\subseteq Z[p^k]$. Assume now that $\phi\colon Z\longrightarrow Z'$ is a surjective homomorphism. Then we obtain an exact sequences
\[0\longrightarrow \ker(\phi)\longrightarrow Z\longrightarrow Z'\longrightarrow 0\] and
\[0\longrightarrow \ker(\phi)[p^k]\longrightarrow Z[p^k]\longrightarrow Z'[p^k] \longrightarrow V\longrightarrow 0\]
for a suitable subgroup $V$ of $\ker(\phi)/p^k\ker(\phi)$. It follows that 
\[r_p^k(Z')=r_p^k(V)+r_p^k(Z)-r_p^k(\ker(\phi)).\]
$\ker(\phi)\subseteq Z$ is cofinitely generated. Therefore $r_p^k(\ker(\phi)/p^k\ker(\phi))\le r_p^k(\ker(\phi))$. In particular, $r_p^k(V)\le r_p^k(\ker(\phi))$ and hence $r_p^k(Z')\le r_p^k(Z)$.
\end{proof}

\begin{lemma}
\label{lemma-ranks}
Let $W,X,Y,Z$ be cofinitely generated abelian groups related by the exact sequence
\[W\longrightarrow X\longrightarrow Y\longrightarrow Z.\]
Then we have
\[\vert r_p^k(X)-r^k_p(Y)\vert \le r_p^k(W)+r_p^k(Z). \]
\end{lemma}
\begin{proof}
For $k=1$ this is \cite[Lemma 3.2]{lim-murty}. Note that we can replace $W$ by $\ker (X\longrightarrow Y)$ and $Z$ by $\textup{im}(Y\longrightarrow Z)$ because clearly, $r_p^k(\ker (X\longrightarrow Y))\le r^k_p(W)$ and $r_p^k(\textup{im}(Y\longrightarrow Z))\le r_p^k(Z)$ (see also Lemma \ref{lem:ranksunderepimorphism}). So we can assume that the sequence
\[0\longrightarrow W\longrightarrow X\longrightarrow Y\longrightarrow Z\longrightarrow 0\]
is exact. 
We can break up this exact sequence into two further sequences
\[0\longrightarrow W \stackrel{\iota}{\longrightarrow} X\longrightarrow C\longrightarrow 0, \]
\[0\longrightarrow C\longrightarrow Y\longrightarrow Z\longrightarrow 0.\]
Restricting to the $p^k$-torsion we obtain
\[0\longrightarrow W[p^k]\longrightarrow X[p^k]\longrightarrow C[p^k]\longrightarrow P\longrightarrow 0\]
for a suitable group $P\subseteq W/p^kW$. 
As $W$ is cofinitely generated, it follows that $r_p^k(P)\le r_p^k(W)$. 
We obtain a second exact sequence
\[0\longrightarrow C[p^k]\longrightarrow Y[p^k]\longrightarrow Q\longrightarrow 0\]
where $Q\subseteq Z[p^k]$. Comparing these two exact sequences yields  
\[r_p^k(X)-r_p^k(Y)=r_p^k(W)-r_p^k(Q)-r_p^k(P).\]
From this the claim is immediate.
\end{proof}

Let $L$ be a number field which contains $K$, and let $\Sigma$ be a finite set of primes of $K$. Recall the definition of $\textup{Cl}(L)$ and $\textup{Cl}_\Sigma(L)$ from Section~\ref{subsection:classnumberformulas}. 
\begin{lemma} \label{lemma:vergleich_sigma_cl} 
Let $\Sigma$ be a finite set of primes of $K$ and let $L$ be a finite extension of $K$. Then
\[\vert r_p^k(\textup{Cl}_\Sigma(L))-r_p^k(\textup{Cl}(L))\vert \le k |\Sigma(L)|, \] 
where $\Sigma(L)$ is the set of primes in $L$ above $\Sigma$.
\end{lemma} 
\begin{proof}
For $k=1$ this is \cite[Lemma 5.2]{lim-murty}.
Consider the exact sequence
\[\Z^{|\Sigma(L)|}\longrightarrow \textup{Cl}(L)\longrightarrow \textup{Cl}_\Sigma(L)\longrightarrow 0.\]
We denote the kernel of the natural projection $\textup{Cl}(L)\longrightarrow \textup{Cl}_\Sigma(L)$ by $C_\Sigma$. Then $C_\Sigma$ is finite and $r_p^k(C_\Sigma)\le k|\Sigma(L)|$. We obtain a short exact sequence of finite abelian groups
\[0\longrightarrow C_\Sigma\longrightarrow \textup{Cl}(L)\longrightarrow \textup{Cl}_\Sigma(L)\longrightarrow 0.\]
Now we apply Lemma \ref{lemma-ranks} with $Z=0$.
\end{proof}
The following lemma, which compares the $p^k$-rank of the fine Selmer group with the $p$-rank of a $p^k$-fine Selmer group, will be needed in Section~\ref{section:main}. 
\begin{lemma} \label{lemma:lim-murty} 
Let $A$ be an abelian variety of dimension $d$ defined over $K$, and let $\Sigma$ be a finite set of primes of $K$ containing $\Sigma_p$ and $\Sigma_{\textup{br}}(A)$. If $p = 2$, then we assume that $K$ is totally imaginary. Let $L/K$ be a finite extension which is contained in $K_\Sigma$. Then 
		\[\vert v_p(|\textup{Sel}_{0,A}(L))[p^k]|)-v_p(|\textup{Sel}_{0,A[p^k],\Sigma}(L)|)\vert \le 2dk(1+|\Sigma(L)|) \] 
		for each integer $k \ge 1$, where $\Sigma(L)$ denotes the set of primes of $L$ above $\Sigma$. 
\end{lemma} 
\begin{proof} 
  This is a special instance of \cite[Lemma~3.1]{KM21}. 
\end{proof}

\subsection{Iwasawa modules} 
In this section we prove auxiliary results on the asymptotic growth of Iwasawa modules over uniform $p$-extensions.  
We start with a slight modification of the result \cite[Theorem~2.1(ii)]{perbet} of Perbet.\footnote{Actually Perbet considered, more generally, pro-$p$-groups which are \emph{$p$-valued}. This includes the uniform $p$-groups (see \cite[p.~81]{Dixon}).} By \cite[Theorem~3.40]{venjakob}, for a uniform group $G$, one can attach to each finitely generated $\Z_p\llbracket G\rrbracket $-module $M$ an elementary $\Z_p\llbracket G \rrbracket $-module ${E = \bigoplus_{i = 1}^s \Z_p\llbracket G \rrbracket /(p^{e_i})}$ which is pseudo-isomorphic to $M[p^\infty]$. In particular, ${\mu_{\Z_p\llbracket G \rrbracket }(M) = \sum_{i=1}^s e_i}$.  
\begin{thm} \label{thm:perbet+}
   Let $K_\infty/K$ be a $p$-adic Lie extension with Galois group $G$. We assume that $G$ is uniform of dimension $l$. Let $M$ be a finitely generated $\Z_p\llbracket G\rrbracket $-module, and let $k \in \N$. Let 
   $E = \bigoplus_{i = 1}^s \Z_p\llbracket G\rrbracket /(p^{e_i})$ be the elementary $\Z_p\llbracket G\rrbracket $-module attached to $M[p^\infty]$ as above. 
   
   Then 
   \[ | v_p(|M_{G_n}/p^k M_{G_n}|) - (r \cdot k + \mu^{(k)}(M)) \cdot p^{ln}| = \Ok(k p^{n (l-1)}),  \] 
   where $\mu^{(k)}(M) = \sum_{i=1}^s \min(k, e_i)$ and $r = \rg_{\Z_p\llbracket G\rrbracket }(M)$. Here the implicit $\Ok$-constant does not depend on $k$ or $n$. 
\end{thm} 
\begin{proof} 
  This result can be proved along the lines of Perbet's proof of \cite[Theorem~2.1(ii)]{perbet}. We therefore only provide a sketch of the proof. 
  
  For every finitely generated $\Z_p\llbracket G\rrbracket $-module $N$, we consider the homology groups 
  \[ H_i(G_n, N) = \textup{Tor}_i^{\Z_p\llbracket G_n\rrbracket }(\Z_p, N), \] 
  $i, n \in \N$. In particular, $H_0(G_n, N) = N_{G_n}$. 
  
  If $N$ is any finitely generated torsion $\F_p\llbracket G\rrbracket $-module (in particular, $p N = \{0\}$), then 
  \begin{align} \label{eq:perbet1} 
     \dim_{\F_p}(H_i(G_n, N)) = \Ok(p^{n(l-1)}) 
  \end{align} 
  for each $i \in \N$ by \cite[Corollaire~2.3]{perbet}. We derive from this fact the following variant of \cite[Corollaire~2.4]{perbet}: let $N$ be a finitely generated $\Z_p\llbracket G\rrbracket $-module such that $N/pN$ is a torsion $\F_p\llbracket G\rrbracket $-module (e.g. this is guaranteed if $N$ is pseudo-null over $\Z_p\llbracket G\rrbracket $, by \cite[Lemme~1.9]{perbet}). Then 
  \begin{align} \label{eq:perbet2} 
     v_p(|H_i(G_n, N/p^kN)|) = \Ok(k p^{n(l-1)})
  \end{align}  
  for each $i \ge 0$, where the implicit constant does not depend on $k$ (or $n$). 
  
  Indeed, taking $G_n$-homology of the exact sequence 
  \[ 0 \longrightarrow p^{j+1} N/ p^k N \longrightarrow p^j N / p^k N \longrightarrow p^jN / p^{j+1} N \longrightarrow 0, \] 
  $j < k$, yields inequalities 
  \[ v_p(|H_i(G_n, p^jN/p^k N)|) \stackrel{\eqref{eq:perbet1}}{\le} v_p(|H_i(G_n, p^{j+1}N/p^k N)|) + C_j \cdot p^{n(l-1)} \] 
  for some constants $C_j$ which do not depend on $n$ or $k$. Therefore 
  \[ v_p(|H_i(G_n, N/p^kN)|) \le (C_1 + \ldots + C_k) \cdot p^{n(l-1)}. \] 
  Since $(p^jN)[p^\infty] = \{0\}$ and therefore $p^jN / p^{j+1}N \cong p^{j+1}N/p^{j+2}N$ for each sufficiently large $j$, the constants $C_j$ are bounded independently of $k$ by some constant $C$, and \eqref{eq:perbet2} follows. 
  
  We deduce the following 
  \begin{prop} \label{prop:pseudo-iso} 
    Let $M$ and $N$ be two finitely generated $\Z_p\llbracket G\rrbracket $-modules which are pseudo-isomorphic. Then 
    \[ | v_p(|M_{G_n}/p^k M_{G_n}|) - v_p(|N_{G_n}/p^k N_{G_n}|)| = \Ok(k p^{n(l-1)})\] 
    for every $k \in \N$, where the implicit constant does not depend on $k$ (or $n$). 
  \end{prop} 
  \begin{proof} 
    Use the arguments from the proof of \cite[Proposition~2.5]{perbet} by replacing \cite[Corollaire~2.4]{perbet} by the estimate~\eqref{eq:perbet2} above. 
  \end{proof} 
  Now we apply the general structure theory of finitely generated $\Z_p\llbracket G\rrbracket $-modules (see \cite[Propositions~6.4 and 7.2]{coates-schneider-sujatha}). 
  Let $M$ be as in the statement of Theorem~\ref{thm:perbet+}, and let $t(M) \subseteq M$ denote the $\Z_p\llbracket G\rrbracket $-torsion submodule. We write ${\tilde{M} = M/ t(M)}$, as in \cite{perbet}. Using Proposition~\ref{prop:pseudo-iso} instead of \cite[Proposition~2.5]{perbet}, the proof of \cite[Corollaire~2.7]{perbet} implies that 
  \begin{align} \label{eq:perbet3} 
     |v_p(|M_{G_n}/p^k M_{G_n}|) - v_p(|t(M)_{G_n}/p^k t(M)_{G_n}|) - v_p(|\tilde{M}_{G_n}/p^k \tilde{M}_{G_n}|)| =  \Ok(k p^{n(l-1)}). 
  \end{align} 
  Now the proofs of \cite[Lemme~2.8 and Lemme~2.9]{perbet}  imply that 
  \begin{align} \label{eq:perbet4} 
    | v_p(|t(M)_{G_n} / p^k t(M)_{G_n}|) -  \mu^{(k)}(M) \cdot p^{ln}| = \Ok(k p^{n(l-1)}), 
  \end{align} 
  because $\mu^{(k)}(t(M)) = \mu^{(k)}(M)$, and 
  \begin{align} \label{eq:perbet5}
    | v_p(|\tilde{M}_{G_n}/p^k \tilde{M}_{G_n}|) - r k p^{ln}| = \Ok(k p^{n(l-1)}). 
  \end{align} 
  Combining the equations~\eqref{eq:perbet3}, \eqref{eq:perbet4} and \eqref{eq:perbet5} proves the theorem. 
\end{proof} 

Letting $k = n$ in the above result and noting that $\mu^{(k)}(M) = \mu(M)$ for $k \gg 0$, we recover Perbet's asymptotic growth formula: 
\begin{cor} \label{cor:perbet+} 
  Let $K_\infty/K$ be a $p$-adic Lie extension with Galois group $G$. We assume that $G$ is uniform of dimension $l$. Let $M$ be a finitely generated $\Z_p\llbracket G\rrbracket $-module. Then 
  \[ | v_p(|M_{G_n}/p^n M_{G_n}|) - (r \cdot n + \mu(M)) p^{ln}| = \Ok(n p^{(l-1)n}), \] 
  where again $r = \rg_{\Z_p\llbracket G\rrbracket }(M)$. 
\end{cor} 

\subsection{Class numbers} 
In this subsection we prove two auxiliary results on the growth of class numbers in multiple $\Z_p$-extensions, which might be of independent interest. We start with a quotient version of the asymptotic class number formula of Cuoco and Monsky (see also Section~\ref{subsection:classnumberformulas}), generalised to $\Sigma$-class groups. 
\begin{thm} \label{thm:ckmm} 
  Let $K_\infty/K$ be a $\Z_p^l$-extension, and let $\Sigma$ be a finite set of primes of $K$. Then 
  \[ |v_p(|Y_\Sigma(K_n)/p^{n} Y_\Sigma(K_n)|) - (m_0 p^{ln} + l_0 n p^{(l-1)n})| = \Ok(p^{(l-1)n}), \] 
  where $m_0, l_0 \in \N$ are the generalised Iwasawa invariants of the $\Lambda_l$-Iwasawa-module $Y_\Sigma(K_\infty) = \varprojlim_n Y_\Sigma(K_n)$. 
\end{thm} 
\begin{proof} 
  The proof is based on the proof of formula~\eqref{eq:cuoco-monsky} given in \cite{cuoco-monsky}; we will heavily use some of the notation from that article. In particular, following Cuoco and Monsky, we let $r(M)$ denote the rank of a finitely generated $\Z_p$-module $M$, and we let $e(M) = v_p(|M[p^\infty]|)$. On the other hand, in order to avoid confusion with the subgroups ${G_n \subseteq G}$ of uniform groups $G$, we have to modify parts of the notation from \cite{cuoco-monsky}. For any finitely generated and torsion $\Lambda_l$-module, we let ${M_{I_n} := M/(I_n \cdot M)}$, where $I_n \subseteq \Lambda_l$ is the ideal generated by $\omega_n(T_1), \ldots, \omega_n(T_l)$, $\omega_n(T_i) := (T_i+1)^{p^n}-1$. More generally, we let ${J_n \supseteq I_n \cdot M}$ denote the submodule of $M$ which is denoted by $A_n'$ in \cite{cuoco-monsky}; the quotient ${M_{J_n} := M/J_n}$ is called $G_n$ in \cite{cuoco-monsky}.  
  
  We start with three auxiliary lemmata. 
  \begin{lemma} \label{lemma:ckmm1} 
     Let $M$ and $M'$ be two finitely generated and torsion $\Lambda_l$-modules, and suppose that $M$ and $M'$ are pseudo-isomorphic. Then 
     \[ |v_p(|M_{J_n}[p^\infty]/p^{n} M_{J_n}[p^\infty]|) - v_p(|(M')_{J_n}[p^\infty]/p^{n} (M')_{J_n}[p^\infty]|) = \Ok(p^{(l-1)n}). \]  
  \end{lemma} 
  \begin{proof} 
     Let $\varphi: M \longrightarrow M'$ be a pseudo-isomorphism. Fixing $n$, we have an induced homomorphism 
     ${\varphi'': M_{J_n}[p^\infty] \longrightarrow (M')_{J_n}[p^\infty]}$ with kernel and cokernel having a valuation of $\Ok(p^{n(l-1)})$ by \cite[Theorem~4.12]{cuoco-monsky} (cf. also the proof of \cite[Lemma~3.3]{cuoco-monsky}). 
     
     Now we consider the induced map $$\tilde{\varphi}: M_{J_n}[p^\infty]/p^{n} M_{J_n}[p^\infty] \longrightarrow (M')_{J_n}[p^\infty]/p^{n}(M')_{J_n}[p^\infty].$$ 
     Then we have a natural surjection 
     \[ \coker(\varphi'') \twoheadrightarrow \coker(\tilde{\varphi}), \] 
     and therefore $v_p(|\coker(\tilde{\varphi})|)$ is again of size $\Ok(p^{(l-1)n})$. It remains to look at the kernel of $\tilde{\varphi}$. We have a natural homomorphism $\ker(\varphi'') \longrightarrow \ker(\tilde{\varphi})$. Let ${\overline{x} \in \ker(\tilde{\varphi})}$, and let $x \in M_{J_n}[p^\infty]$ be a representative of $\overline{x}$. Then we can write ${\varphi''(x) = p^{n} y}$ for some ${y \in (M')_{J_n}[p^\infty]}$. If $y$ lies in the image of $\varphi''$, then there exists a $y'\in M_{J_n}$ such that ${x - p^{n}y' \in \ker(\varphi'')}$, which is equivalent to $\overline{x}$ lying in the image of $\ker(\varphi'')$ under the canonical projection ${\pi \colon M_{J_n}[p^\infty] \longrightarrow M_{J_n}[p^\infty]/p^n M_{J_n}[p^\infty]}$. 
     To every ${\overline{x}\in \ker(\tilde{\varphi})/\pi(\ker(\varphi''))}$ we can associate an element $y$ in the cokernel of $\varphi''$ and this correspondence is one-to-one.
     
     Indeed, assume that we have two classes $\overline{x}_1$ and $\overline{x}_2$ in $\ker(\tilde{\varphi})/\pi(\ker(\varphi''))$ such that $y_1-y_2=\varphi''(z)$, for some element $z\in M_{J_n}[p^\infty]$. Then we obtain
     \[\varphi''(x_1-x_2-p^nz)=0.\]
     Thus, $\overline{x}_1=\overline{x}_2$ and the correspondence is indeed injective. 
     As $v_p(\coker(\varphi''))$ is of size $\mathcal{O}(p^{n(l-1)})$, the same holds for $\ker(\tilde{\varphi})/\pi(\ker(\varphi''))$. 
     The above observations imply that ${v_p(|\ker(\tilde{\varphi}))|) = \Ok(p^{(l-1)n})}$. 
  \end{proof} 
  
  \begin{lemma} \label{lemma:ckmm2} 
     Let $M$ be a finitely generated and torsion $\Lambda_l$-module. In this lemma, we consider the special case $M_{I_n}$. 
     
     Assume that $r(M_{I_n}) = \Ok(p^{(l-2)n})$. Then 
     \[ v_p(|M_{I_n}[p^\infty]/p^{n} M_{I_n}[p^\infty]|) = (m_0 p^n + l_0 n + \Ok(1)) p^{(l-1)n}, \]  where $m_0, l_0 \in \N$ are the generalised Iwasawa invariants attached to $M$. 
  \end{lemma} 
  \begin{proof} 
    By Lemma~\ref{lemma:ckmm1}, we are reduced to consider an elementary $\Lambda_l$-module, and we can assume that $M$ is cyclic. If $M = \Lambda_l/(p^v)$ for some $v \in \N$, then the result is immediate. Therefore suppose that $M = \Lambda_l/(F)$ for some element $F$ which is coprime with $p$. It follows from \cite[Theorem~2.5]{cuoco-monsky} that 
    \[ v_p(|M_{I_n}[p^\infty]|) = (m_0 p^n + l_0 n + \Ok(1)) p^{(l-1)n} . \] 
    Therefore it suffices to prove that 
    \begin{align} \label{eq:diff} 
      v_p(|p^n M_{I_n}[p^\infty]|) = v_p(|M_{I_n}[p^\infty]|) - v_p(|M_{I_n}[p^\infty]/p^n M_{I_n}[p^\infty]|) = \Ok(p^{(l-1)n}). 
    \end{align} 
    To this purpose, we use the following auxiliary construction from \cite{cuoco-monsky}. 
    
    We let $W_n$ denote the set of $l$-tuples $\underline{\zeta} = (\zeta_1, \ldots, \zeta_l)$, where each $\zeta_i$ is a root of unity of order dividing $p^n$ (contained in some fixed algebraic closure of $\Q_p$). For any ${\underline{\zeta} = (\zeta_1, \ldots, \zeta_l) \in W_n}$, the cyclotomic ring ${\Z_p[\underline{\zeta}] := \Z_p[\zeta_1, \ldots, \zeta_l]}$ is a free $\Z_p$-module of rank $\varphi(p^r)$, where $\varphi$ denotes Euler's totient function and $p^r$ is the maximum of the orders of the $\zeta_i$. 
    
    Let 
    \[ \psi_n \colon \Lambda_l/I_n \longrightarrow \bigoplus_{W_n/\sim} \Z_p[\underline{\zeta}] \] 
    be the homomorphism induced by mapping ${g \in \Lambda_l}$ to ${g(\zeta_1, \ldots, \zeta_l) \in \Z_p[\zeta_1, \ldots, \zeta_l]}$. Here the sum runs over a set of representatives of the elements of $W_n$ modulo the action of the absolute Galois group of $\Q_p$. 
    
    By \cite[Theorem~2.2]{cuoco-monsky}, the map $\psi_n$ is injective for every $n \in \N$. We define $\Z_p$-modules $X_n$ by the following commutative diagrams. 
    \[ \xymatrix{\Lambda_l/I_n \ar[r]^-{\psi_n} \ar[d]^{\cdot F} & \bigoplus_{W_n/\sim} \Z_p[\underline{\zeta}] \ar[d]\\ 
    \Lambda_l/I_n \ar[r]^-{\psi_n} \ar[d] & \bigoplus_{W_n/\sim} \Z_p[\underline{\zeta}] \ar[d] \\ 
    M_{I_n} & X_n}\] 
    Here the right vertical map is multiplication by ${F(\underline{\zeta}-1) = F(\zeta_1 -1, \ldots, \zeta_l-1)}$ in the summand $\Z_p[\underline{\zeta}]$, respectively. 
    
    Since $r(M_{I_n}) = \Ok(p^{(l-2)n})$ by assumption, it follows from \cite[Theorem~2.2 and Lemma~2.4]{cuoco-monsky} (cf. also the proof of \cite[Theorem~2.5]{cuoco-monsky}) that 
    \[ |e(M_{I_n}) - e(X_n)| = \Ok(n p^{(l-2)n}). \]  
    As in the proof of Lemma~\ref{lemma:ckmm1}, we may deduce that 
    \[ |v_p(|M_{I_n}[p^\infty]/p^n M_{I_n}[p^\infty]|) - v_p(|X_n[p^\infty]/p^n X_n[p^\infty]|)| = \Ok(n p^{(l-2)n)}). \] 
    In view of \eqref{eq:diff}, we are reduced to proving that 
    \[ v_p(|p^n X_n[p^\infty]|) = \Ok(p^{(l-1)n}). \] 
    Note that 
    \[ X_n[p^\infty] = \bigoplus_{\substack{W_n/\sim\\ F(\underline{\zeta}-1) \ne 0}} \Z_p[\underline{\zeta}]/F(\underline{\zeta}-1), \] 
    where the sum runs over all representatives ${\underline{\zeta} = (\zeta_1, \ldots, \zeta_l)}$ such that 
    $$F(\underline {\zeta}-1) = F(\zeta_1-1, \ldots, \zeta_l-1) \ne 0. $$
    Therefore 
    \[ |X_n[p^\infty]| = \sum_{\substack{\underline{\zeta} \in W_n/\sim: \\ F(\underline{\zeta}-1) \ne 0}} \; \sum_{\zeta' \sim \zeta}  v_p(F(\underline{\zeta'}-1)) \] 
    (cf. also the proof of \cite[Theorem~2.5]{cuoco-monsky}) and 
    \[ |p^n X_n[p^\infty]| \le \sum_{\substack{\underline{\zeta} \in W_n/\sim: \; F(\underline{\zeta}-1) \ne 0, \\ v_p(F(\underline{\zeta}-1)) \ge n}} \; \sum_{\zeta' \sim \zeta} v_p(F(\underline{\zeta'}-1)). \] 
    It follows from \cite[Theorem~2.5]{monsky} that $v_p(F(\underline{\zeta}-1))$ is bounded independently of $n$. Therefore the latter sum is empty for sufficiently large $n$. This shows that actually 
    \[ |p^n X_n[p^\infty]| = \Ok(1)\] 
    and concludes the proof of Lemma~\ref{lemma:ckmm2}. 
  \end{proof} 
  \begin{lemma} \label{lemma:ckmm3} 
    Let $M$ be a finitely generated and torsion $\Lambda_l$-module, and consider now the more general quotients $M_{J_n}$, where ${J_n \supseteq I_n \cdot M}$ is defined as above. Then 
    \[ v_p(|M_{J_n}[p^\infty]/p^{n}M_{J_n}[p^\infty]|) = (m_0p^n + l_0 n + \Ok(1)) p^{(l-1)n}, \] 
    where again $m_0 = m_0(M)$ and $l_0 = l_0(M)$. 
  \end{lemma} 
  \begin{proof} 
     By Lemma~\ref{lemma:ckmm1}, we can assume that the underlying module $M$ is $p$-elementary and cyclic, so $M = \Lambda_l/(\psi^s)$ for some $\psi \in \Lambda_l$ and $s \in \N$. 
     
     If $\psi$ is not special in the sense of Cuoco and Monsky, then 
     $$\rg_{\Z_p}(M/I_n M) = \Ok(p^{(l-2)n})$$ 
     by \cite[Theorem~3.13]{cuoco-monsky}. In this case, we can define $J_n = I_n \cdot M$, and therefore the assertion follows from Lemma~\ref{lemma:ckmm2}. 
     
     If $\psi$ is a special prime, then it follows from the proof of \cite[Theorem~4.13]{cuoco-monsky} that $M_{J_n}$ can be identified with a quotient of $R_n\llbracket T_2, \ldots, T_l\rrbracket $, where $$R_n = \Z_p\llbracket X\rrbracket /(g^s, h_n)$$ 
     is a finite ring. Here either 
     \[ g = X \quad \text{ and } \quad h_n = \frac{\omega_n(X)}{X}, \] 
     or 
     \[ g  = \frac{\omega_r(X)}{\omega_{r-1}(X)} \quad \text{ and } \quad h_n = \frac{\omega_n(X)}{\omega_r(X)}. \] 
     
     By the classical theory of finitely generated torsion $\Z_p\llbracket X\rrbracket $-modules, one can show that the exponent of $R_n$ is equal to $p^{n+c}$ for some fixed $c \in \N$. Indeed, letting $V = \Z_p\llbracket X\rrbracket /(g^s)$, we have $h_{n+1}V = p h_n V$ for all sufficiently large $n \in \N$ (cf. the proof of \cite[Lemma~13.18]{wash}; this holds for both cases described above). Therefore 
     \begin{align} \label{eq:G_n} 
         |v_p(|M_{J_n}[p^\infty]/p^{n} M_{J_n}[p^\infty]|) - e(M_{J_n})| = \Ok(p^{(l-1)n}) 
     \end{align} 
     since $\rg_p(M_{J_n}[p^\infty]) = \Ok(p^{(l-1)n})$. 
     
     The result now follows from \cite[Theorem~4.13]{cuoco-monsky}.  
  \end{proof} 
  Now we return to the proof of Theorem~\ref{thm:ckmm}. We first note that it follows from \cite[Theorem~5.12]{cuoco-monsky} that 
  \[ |v_p(|Y(K_n)|) - v_p(|(Y(K_\infty)/A_n')[p^\infty]|)| = \Ok(n), \] 
  where $A_n'$ is defined as in \cite[Definition 5.5]{cuoco-monsky} (note that here $A'_n$ is induced from a specific structure related to class groups). Actually the same proof works more generally for $\Sigma$-class groups; the only change to be made is to replace the inertia subgroups used in the definition of $A_n'$ by decomposition subgroups and to add the decomposition groups for the primes $v \in \Sigma$ which do not ramify in $K_\infty/K$. 
  
  Let $B_n$ be defined as in \cite[Definition 5.3]{cuoco-monsky}. By \cite[Lemma~4.15]{cuoco-monsky}, which is purely algebraic and valid also for $Y_\Sigma(K_\infty)$ instead of $Y(K_\infty)$, we have 
  \[ |e(Y(K_\infty)/A_n') - e(Y(K_\infty)/B_n)| = \Ok(n).\] Again, we can derive a quotient version of this. Therefore the difference between 
  \[ v_p(|(Y(K_\infty)/A_n')[p^\infty]/p^{n}(Y(K_\infty)/A_n')[p^\infty]|)\] 
  and \[ v_p(|(Y(K_\infty)/B_n)[p^\infty] / p^{n} (Y(K_\infty)/B_n)[p^\infty]|) \] 
  is $\Ok(n)$, 
  and an analogous estimate holds for $\Sigma$-class groups. 
  
  On the other hand, it follows from the proof of \cite[Theorem~5.4]{cuoco-monsky} that $Y_\Sigma(K_\infty)/B_n$ is finite for every $n \in \N$, i.e. ${(Y_\Sigma(K_\infty)/B_n)[p^\infty] = Y_\Sigma(K_\infty)/B_n}$, and that 
  \[ | v_p(|Y_\Sigma(K_n)|) - v_p(|Y_\Sigma(K_\infty)/B_n|) | = \Ok(1).\] 
  More precisely, we have embeddings $Y_\Sigma(K_\infty)/B_n \hookrightarrow Y_\Sigma(K_n)$ with cokernels of bounded order. Therefore  
  \[ | v_p(|(Y_\Sigma(K_\infty)/B_n)/p^{n}(Y_\Sigma(K_\infty)/B_n)|) - v_p(|Y_\Sigma(K_n)/p^{n}Y_\Sigma(K_n)|)| = \Ok(1).  \] 
  We have shown that 
  \[ |v_p(|Y_\Sigma(K_n)/p^{n}Y_\Sigma(K_n)|) - v_p(|(Y_\Sigma(K_\infty)/A_n')[p^\infty]/p^{n}(Y_\Sigma(K_\infty)/A_n')[p^\infty]|)| = \Ok(n). \] 
  Since $A_n' = J_n$ in our notation, Theorem~\ref{thm:ckmm} follows by combining this estimate with Lemma~\ref{lemma:ckmm3}. 
\end{proof} 

The second result which we prove in this subsection compares the $\mu$-invariant of a $\Z_p^l$-extension with the $\mu$-invariants of suitable $\Z_p^{l-1}$-subextensions. 
\begin{prop} \label{lemma:mu_going_down} 
  Let $K_\infty/K$ be a $\Z_p^l$-extension, $l \ge 2$, with Galois group $G$, and let $\Sigma$ be a finite set of primes of $K$ which contains the primes ramifying in $K_\infty$. We consider the set $\mathcal{H}$ of $\Z_p^{l-1}$-subextensions $\LL_\infty \subseteq K_\infty$ of $K$. 
  
  For any $\LL_\infty \in \mathcal{H}$, we let ${H = \Gal(\LL_\infty/K)}$, i.e.  ${H = G/\Gal(K_\infty/\LL_\infty)}$, and we consider the two elementary modules 
  \[ \bigoplus_{i=1}^s \Z_p\llbracket G\rrbracket /(p^{e_i}) \sim Y(K_\infty)[p^\infty] \quad \text{and} \quad \bigoplus_{i = 1}^{s'} \Z_p\llbracket H\rrbracket /(p^{f_i}) \sim Y(\LL_\infty)[p^\infty]. \] 
  If the decomposition subgroup $D_v(K_\infty/K)$ of $G$ at every prime $v \in \Sigma$ has dimension at least one, then $s' = s$ and $f_i = e_i$ for each $i \in \N$ holds for all but finitely many $\LL_\infty \in \mathcal{H}$. 
\end{prop}
\begin{proof}
 For any $\LL_\infty \in \mathcal{H}$, we let $\sigma$ be a topological generator of ${\Gal(K_\infty/\LL_\infty) \cong \Z_p}$, and write $S = \sigma - 1$. We divide the proof into several small lemmas. 
\begin{lemma}
   The $\Z_p\llbracket H\rrbracket $-module 
  $$Y(K_\infty)[p^\infty]/(S \cdot Y(K_\infty)[p^\infty])$$ 
  is pseudo-isomorphic to $\bigoplus_{i = 1}^s \Z_p\llbracket H\rrbracket /(p^{e_i})$ for all but finitely many $\LL_\infty \in \mathcal{H}$. 
\end{lemma}
\begin{proof}  
  Let ${W = Y(K_\infty)[p^\infty]}$. Then $W$ is $\Z_p\llbracket G\rrbracket $-torsion, and we have a pseu\-do-\-iso\-mor\-phism ${\varphi: E_W \longrightarrow W}$ with cokernel $Z$, where $E_W$ is an elementary $\Lambda_l$-module (see Section~\ref{section:notation_extensions} for more details). Since $E_W$ does not contain any non-trivial pseudo-null submodules, $\varphi$ is injective. Starting from the exact sequence 
  $$ 0 \longrightarrow E_W \longrightarrow W \longrightarrow Z \longrightarrow 0,$$ 
  the Snake Lemma yields an exact sequence 
  $$ Z[S] \longrightarrow E_W/S E_W \longrightarrow W/SW \longrightarrow Z/SZ \longrightarrow 0.$$ 
  It is sufficient to prove that $Z[S]$ and $Z/SZ$ are pseudo-null over $$\Z_p\llbracket H\rrbracket  \cong \Z_p\llbracket G\rrbracket /(S)$$ 
  for all but finitely many $\LL_\infty \in \mathcal{H}$ (i.e. for all but finitely many $S$). 
  We show that $\textup{ht}_{\Z_p\llbracket G\rrbracket }(\textup{Ann}(Z) + (S)) \ge 3$ for all but finitely many $S$, where $\textup{Ann}(Z) \subseteq \Z_p\llbracket G\rrbracket $ denotes the annihilator ideal. From this the claim is immediate. 
  
  In order to do so, we use an argument a variant of which has been used in the proof of \cite[Lemma~3.1]{local_max}. We choose minimal prime ideals $\p_S \supseteq \textup{Ann}(Z) + (S)$ for each $\LL_\infty \in \mathcal{H}$. If $\textup{ht}(\textup{Ann}(Z) + (S)) = 2$, then $\p_S$ is a minimal prime ideal of $\textup{Ann}(Z)$ (note that $\textup{ht}(\textup{Ann}(Z)) \ge 2$ because $Z$ is pseudo-null). Since ${\Z_p\llbracket G\rrbracket  \cong \Lambda_l}$ is a Noetherian domain, there exist only finitely many minimal prime ideals of $\textup{Ann}(Z)$. 
  
  Now let $\LL_\infty, \LL_\infty' \in \mathcal{H}$ be two different $\Z_p^{l-1}$-extensions of $K$ contained in $K_\infty$, and let $S = \sigma - 1$ and $S' = \sigma' - 1$ be the corresponding elements of $\Z_p\llbracket G\rrbracket $. Suppose that $\textup{ht}(\p_S) = \textup{ht}(\p_{S'}) = 2$. We show that $\p_S \ne \p_{S'}$. Indeed, if $\p_S = \p_{S'}$, then this prime ideal contains $S$, $S'$ and some power of $p$ (recall that $Z = \textup{coker}(\varphi)$ is a quotient of $W = Y(K_\infty)[p^\infty]$). But this is possible only if $\textup{ht}(\p_S) \ge 3$. 
  
  We have shown that for all but finitely many $S$, the height of the $\Z_p\llbracket G\rrbracket $-ideal $\textup{Ann}(Z) + (S)$ is at least equal to 3. This shows that we have pseudo-isomorphisms 
  $$ Y(K_\infty)[p^\infty]/(S \cdot Y(K_\infty)[p^\infty]) \sim \bigoplus_{i = 1}^s \Z_p\llbracket H\rrbracket /(p^{e_i}) $$ 
  of $\Z_p\llbracket H\rrbracket $-modules for all but finitely many $\LL_\infty \in \mathcal{H}$. 
  \end{proof}
  Now we compare the two $\Z_p\llbracket H\rrbracket $-modules $Y(K_\infty)[p^\infty]/(S \cdot Y(K_\infty)[p^\infty])$ and $(Y(K_\infty)/(S \cdot Y(K_\infty)))[p^\infty]$. Let $k \in \N$ be large enough such that $Y(K_\infty)[p^\infty]$ and $(Y(K_\infty)/(S \cdot Y(K_\infty))[p^\infty]$ are both annihilated by $p^k$. Then the $p^\infty$-torsion of the above modules coincides with the $p^k$-torsion. Starting from an exact sequence 
  $$ 0 \longrightarrow S \cdot Y(K_\infty) \longrightarrow Y(K_\infty) \longrightarrow Y(K_\infty)/(S \cdot Y(K_\infty)), $$ 
  the Snake Lemma yields exact sequences  
  \[ 0 \longrightarrow SY(K_\infty)[p^k] \longrightarrow Y(K_\infty)[p^k] \longrightarrow (Y(K_\infty)/SY(K_\infty))[p^k] \] 
  and 
  \[ (Y(K_\infty)/SY(K_\infty))[p^k] \stackrel{\alpha}{\longrightarrow} SY(K_\infty)/p^kSY(K_\infty) \stackrel{\beta}{\longrightarrow} Y(K_\infty)/p^kY(K_\infty). \] 
  In particular, we obtain an injection 
  $$Y(K_\infty)[p^\infty]/(S \cdot Y(K_\infty))[p^\infty] \hookrightarrow (Y(K_\infty)/(S \cdot Y(K_\infty)))[p^\infty]; $$ it remains to prove that the cokernel of this map is pseudo-null as a module over ${\Z_p\llbracket H\rrbracket  \cong \Z_p\llbracket G\rrbracket /(S)}$  
  for all but finitely many possible choices of $S$. 
  
  By the Snake Lemma, the above cokernel is isomorphic to the image of the map $\alpha$, which is equal to the kernel of the map 
  $$ \beta: S Y(K_\infty)/p^k S Y(K_\infty) \longrightarrow Y(K_\infty)/p^k Y(K_\infty). $$ 
  This kernel equals $(S Y(K_\infty) \cap p^k Y(K_\infty))/p^k S Y(K_\infty)$. 
  \begin{lemma} 
    If $k \in \N$ is sufficiently large, then 
    $$(S Y(K_\infty) \cap p^k Y(K_\infty))/p^k S Y(K_\infty)$$ 
    is pseudo-null over $\Z_p\llbracket G\rrbracket /(S)$ for all but finitely many choices of ${\LL_\infty \in \mathcal{H}}$. 
  \end{lemma} 
  \begin{proof} 
    Since no prime ${v \in \Sigma}$ is completely split in $K_\infty/K$ by assumption, it follows from \cite[Theorem~6.1]{ochi-venjakob} that $Y(K_\infty)$ is a $\Z_p\llbracket G\rrbracket $-torsion module. 
    
    Since the quotient $(S Y(K_\infty) \cap p^k Y(K_\infty))/p^k S Y(K_\infty)$ is of course annihilated by $p^k$, we look for a further annihilator in ${\Lambda_l = \Z_p\llbracket G\rrbracket }$ which is coprime with $p$ and remains coprime when we reduce it modulo $S$. In other words, we look for an annihilator which is not contained in $(S, p)$. Let ${F := F_{Y(K_\infty)}}$ be the characteristic power series of the $\Lambda_l$-module $Y(K_\infty)$, and let ${f_Y = F_Y/p^{m_0(Y(K_\infty))}}$. Since the image of $f_Y$ under the canonical projection ${\Lambda_l \twoheadrightarrow \Lambda_l/p \Lambda_l}$ is divisible by at most finitely many irreducible factors, $f_Y \not\in (p, S)$ for all but finitely many choices of $S$. Now fix any $S$ such that $f_Y \not\in (p, S)$. It follows from (the proof of) \cite[Theorem~3.3]{monsky-annihilator} that there exists $h \in \Lambda_l$ such that $h \not\in (p,S)$ and such that ${p^s h \cdot F_Y = p^{s + m_0(Y(K_\infty))} h \cdot f_Y}$ annihilates $Y(K_\infty)$ for some $s \in \N$ (the proof in \cite{monsky-annihilator} is given for $\Lambda_2$-modules, but the argument remains valid for arbitrary $l$, cf. also Lemmas~4.36 and 4.39 in \cite{diss}). If ${k \ge s + m_0(Y(K_\infty))}$, then the submodule $p^k Y(K_\infty)$ of $Y(K_\infty)$ is annihilated by $h f_Y$. Therefore the quotient ${(S Y(K_\infty) \cap p^k Y(K_\infty))/p^k S Y(K_\infty)}$ is pseudo-null for all these $S$. 
  \end{proof} 
  We have shown that $(Y(K_\infty)/S Y(K_\infty))[p^\infty] \sim \bigoplus_{i=1}^s \Z_p\llbracket H\rrbracket /(p^{e_i})$ for all but finitely many $\LL_\infty \in \mathcal{H}$. 
  
  Now fix $\LL_\infty \in \mathcal{H}$, and suppose for the moment that $\dim(D_v(\LL_\infty/K)) \ge 1$ for each $v \in \Sigma$, where ${D_v(\LL_\infty/K) \subseteq \Gal(\LL_\infty/K)}$ denote the decomposition subgroups. We let $X \subseteq H(K_\infty)$ be the maximal subextension which is abelian over $\LL_\infty$. Then we have an exact sequence 
  \[ 0 \longrightarrow \sum_{w \in \Sigma(\LL_\infty)} I_w(X/\LL_\infty) \longrightarrow \Gal(X/\LL_\infty) \longrightarrow \Gal(H(\LL_\infty)/\LL_\infty) \longrightarrow 0,\] 
  where $I_w(X/\LL_\infty) \subseteq \Gal(X/\LL_\infty)$ denotes the inertia subgroup. Since each inertia subgroup has $\Z_p$-rank at most one (because $X/K_\infty$ is unramified), each direct sum $I_v := \sum_{w \mid v, v \in \Sigma} I_w(X/\LL_\infty)$ is a finitely generated $\Z_p\llbracket H/D_v(\LL_\infty/K)\rrbracket $-module (here $w$ runs over the primes of $\LL_\infty$ dividing $v$, respectively). If $\dim(D_v(\LL_\infty/K)) \ge 1$, then  
  it follows that $\mu_{\Z_p\llbracket H\rrbracket }(I_v) = 0$. If this holds for each $v$, then 
  $s' = s$ and $e_i = f_i$ for all $i$ because 
  \[ Y(K_\infty)/(S \cdot Y(K_\infty))[p^\infty] \cong \Gal(X/K_\infty)[p^\infty] = \Gal(X/\LL_\infty)[p^\infty]\]
  (here $S$ is defined as in the first part of the proof) and 
  $$\Gal(H(\LL_\infty)/\LL_\infty)[p^\infty] \cong Y(\LL_\infty)[p^\infty]. $$ 
  
Hence, the proof will be complete once we have proven the following lemma.
\begin{lemma}
   Suppose that $\dim(D_v(K_\infty/K)) \ge 1$ for each $v \in \Sigma$. Then there exist only finitely many $\LL_\infty \in \mathcal{H}$ such that $\dim(D_v(\LL_\infty/K)) = 0$ for some $v \in \Sigma$.
\end{lemma} 
\begin{proof}
  
 If $\dim(D_v(K_\infty/K)) \ge 2$ for some $v \in \Sigma$, then it is obvious that 
  \[ \dim(D_v(\LL_\infty/K)) \ge \dim(D_v(K_\infty/K)) - 1 \ge 1. \] 
  It remains to consider the case $\dim(D_v(K_\infty/K)) = 1$. We will prove that there are only finitely many $\LL_\infty \in \mathcal{H}$ such that the primes of $\Sigma(\LL_\infty)$ are not totally split in $K_\infty$. 
  
  Since $\dim(D_v(K_\infty/K)) = 1$, it follows from \cite[Theorem~4.8]{Dixon} that $\dim(\Gal(K_\infty^{D_v(K_\infty/K)}/K)) = l-1$, where we recall that $l = \dim(G)$. If $v$ is finitely split in the $\Z_p$-extension $K_\infty/\LL_\infty$, then 
  \begin{eqnarray*} \dim(\Gal((K_\infty^{D_v(K_\infty/K)} \cdot \LL_\infty)/K) & = & \dim(\Gal(\LL_\infty/K)) = l - 1 \\ & = & \dim(\Gal(K_\infty^{D_v(K_\infty/K)}/K)), \end{eqnarray*} 
  i.e. $\dim(\Gal(\LL_\infty/(\LL_\infty \cap K_\infty^{D_v(K_\infty/K)}))) = 0$ by \cite[Theorem~4.8]{Dixon}. Since $H[p^\infty] = \{0\}$ by assumption, it follows that $\Gal(\LL_\infty/(\LL_\infty \cap K_\infty^{D_v(K_\infty/K)})) = \{0\}$, i.e. $\LL_\infty \subseteq K_\infty^{D_v(K_\infty/K)}$. 
  
  Now suppose that there exists another subextension $\tilde{\LL}_\infty \in \mathcal{H}$ of $K_\infty$ such that $v$ is finitely split in $K_\infty/\tilde{\LL}_\infty$. Since $\LL_\infty \cap \tilde{\LL}_\infty \ne \LL_\infty$ and $H[p^\infty] = \{0\}$, it follows that 
  \[ \dim(\Gal((\LL_\infty \cdot \tilde{\LL}_\infty)/K)) > \dim(\Gal(\tilde{\LL}_\infty/K)). \] 
  Therefore $\LL_\infty \cdot \tilde{\LL}_\infty = K_\infty$ since $G[p^\infty] = \{0\}$. But both $\LL_\infty$ and $\tilde{\LL}_\infty$ are contained in $K_\infty^{D_v(K_\infty/K)}$ by assumption, which would imply that the prime $v$ is totally split in $K_\infty/K$, thus yielding a contradiction. 
  
  This shows that $\dim(D_v(\LL_\infty/K)) \ge 1$ for each $v \in \Sigma$ holds for all but at most $|\Sigma|$ elements $\LL_\infty \in \mathcal{H}$. \end{proof} 
  This also concludes the proof of Proposition~\ref{lemma:mu_going_down}. 
\end{proof} 

\begin{cor} \label{cor:_mu_going_up} 
  Let $K_\infty/K$ be a $\Z_p^l$-extension which contains the cyclotomic $\Z_p$-extension $K_\infty^c$ of $K$. If ${\mu(K_\infty^c/K) = 0}$, then also $\mu(K_\infty/K) = 0$. 
\end{cor} 
\begin{proof} 
  We proceed via induction over the dimension $l$ of $\Gal(K_\infty/K)$. Since no prime of $K$ splits completely in $K_\infty^c/K$, the decomposition groups $D_v(\LL_\infty/K)$ have dimensions at least one for any intermediate $\Z_p^{l-1}$-extension ${K_\infty^c \subseteq \LL_\infty \subseteq K_\infty}$.  Therefore 
  \[ \mu_{\Z_p\llbracket \Gal(\LL_\infty/K)\rrbracket }(\LL_\infty/K) = \mu(K_\infty^c/K) = 0 \] 
  for each such $\LL_\infty$, by the proof of Proposition~\ref{lemma:mu_going_down}. 
\end{proof} 

The following lemma generalises a well-known result of Iwasawa (see \cite[Theorem~2]{iwasawa-mu}) and will be used in the proof of Theorem~\ref{thm:Conj_A}. 
\begin{lemma} \label{lemma:mu_iwasawa_allgemein}
  Let $K_\infty/K$ be a uniform $p$-adic Lie extension of dimension $l$, with Galois group $G$. Let $K'/K$ be a cyclic extension of degree $p$, and suppose that the dimension of the decomposition group ${D_v(K_\infty/K)}$ is at least one for each prime $v$ of $K$ which ramifies in $K'$. Let $K_\infty' = K' \cdot K_\infty$ and $G' = \Gal(K_\infty'/K')$. Then 
  \[ \rg_{\Z_p\llbracket G\rrbracket }(Y(K_\infty)) = \mu_{\Z_p\llbracket G\rrbracket }(Y(K_\infty)) = 0 \] 
  if and only if 
  \[ \rg_{\Z_p\llbracket G'\rrbracket }(Y(K_\infty')) = \mu_{\Z_p\llbracket G'\rrbracket }(Y(K_\infty')) = 0.\] 
\end{lemma} 
\begin{proof} 
   In view of \cite[formula~(4) on p.~7]{iwasawa-mu}, we have 
   \[ \rg_p(Y(K_n)) - 1 \le \rg_p(Y(K_n')) \le p \cdot (\rg_p(Y(K_n)) + |\Sigma_{\textup{ram}}(K_n'/K_n)|)\] 
   for every $n \in \N$, where $\Sigma_{\textup{ram}}(K_n'/K_n)$ denotes the set of primes of $K_n$ which ramify in $K_n'$. 
   
   By the decomposition hypothesis, \cite[Lemma~4.1]{limwildkernel} implies that 
   \[ |\Sigma_{\textup{ram}}(K_n'/K_n)| = \mathcal{O}(p^{n(l-1)}). \]
   On the other hand, it follows from \cite[Proposition~3.2]{perbet} and Theorem~\ref{thm:perbet+} (applied with $k = 1$) that 
   \[ |\rg_{p}(Y(K_n)) - (\rg_{\Z_p\llbracket G\rrbracket }(Y(K_\infty)) + \rg_{\F_p\llbracket G\rrbracket }(Y(K_\infty)[p])) \cdot p^{nl}| = \Ok(p^{n(l-1)}), \] 
   and a similar formula holds for $K_n' \subseteq K_\infty'$. This proves the lemma. 
\end{proof} 

\section{Control Theorems} \label{section:6} 
We start with several auxiliary lemmas on local cohomology groups. 

\begin{lemma}
\label{lem:zpext}
Let $K$ be a field of characteristic zero and let $K_\infty/K$ be a $\Z_p^l$-extension ($l\ge 2$), whose Galois group $G$ acts on a $\Z_p$-torsion module $M$ with finite $p$-rank. Assume that $M^G$ is finite. Then we can always find a $\Z_p$-extension $L$ of $K$ contained in $K_\infty$ such that ${M(L) := M^{\Gal(K_\infty/L)}}$ is finite.
\end{lemma}
\begin{proof}
Let $M^\vee$ be the Pontryagin dual of $M$. Then $M^\vee$ is a finitely generated $\Z_p$-module. We denote by $\textup{Ann}(M^\vee) \subseteq \Lambda_l$ the annihilator of $M^\vee$ as a $\Z_p\llbracket G\rrbracket $-module. If $\textup{Ann}(M^\vee)$ contains a power of $p$, then $M^\vee$ and $M$ are finite and there is nothing to prove. Thus, let us assume that $M^\vee$ has positive $\Z_p$-rank.  
As $M^G$ is finite, we know that 
$$ \textup{ht}(\textup{Ann}(M^\vee) + (T_1, \ldots, T_l)) = l+1 > \textup{ht}(\textup{Ann}(M^\vee)). $$ 
Now we can choose topological generators $\tau_1,\dots, \tau_l$ of ${G \cong \Z_p^l}$ such that $$\textup{ht}(\textup{Ann}(M^\vee)+(T_2,\dots T_l))>\textup{ht}(\textup{Ann}(M^\vee)), $$ 
where $T_i = \tau_i - 1$, as usual (in fact, this holds for infinitely many choices of the $T_i$). In particular, $M^\vee/(\textup{Ann}(M^\vee)+(T_2,\dots ,T_l))$ is finite. Let $H$ be the group generated by $\tau_2,\ldots, \tau_l$, and define $L=K_\infty^H$. Then $L$ is a $\Z_p$-extension of $K$ such that $M(L)$ is finite.
\end{proof}

We need the above lemma in order to prove the following
\begin{lemma}
\label{cohomologyisfinite}
Keep the assumptions from Lemma \ref{lem:zpext}. Then both $H^1(G, M)$ and $H^2(G,M)$ are finite.
\end{lemma}
\begin{proof}
   We prove the claim via induction on $l$. If $l=1$, then the theory of Herbrand quotients tells us that $H^1(G,M)$ and $H^2(G,M)$ are bounded by the size of $M^{G}$, which is finite by assumption.
   Let us now assume that the claim has already been proved for all $\Z_p^{l-1}$-extensions. Consider a $\Z_p^l$-extension with Galois group $G$. By Lemma \ref{lem:zpext} we can find a $\Z_p$-extension $L$ of $K$ such that $M(L)$ is finite. Let $H=\Gal(K_\infty/L)$. Then $H$ has a well-defined action on $M$ and $M^H$ is finite. So we can apply the induction hypothesis on $H$ and obtain that $H^1(H,M)$ and $H^2(H,M)$ are finite. Using the exact sequences $(1\le i\le 2)$
   \[H^i(G/H,M^H)\to H^i(G,M)\to H^i(H,M)\]
   reduces the proof to showing that $H^i(G/H,M^H)$ is finite. But this is just the case $l=1$ as $G/H\cong \Z_p$ and $(M^H)^{G/H}=M^G$ is finite.
\end{proof}
\begin{lemma}
\label{ldifferentp}
Let $A$ be an abelian variety defined over a finite extension $K$ of $\Q_l$ for some $l\neq p$. Let $K_{\infty}/K$ be a uniform $p$-adic Lie extension. Then the following two assertions hold.
\begin{itemize}
    \item[(a)] If $K_\infty/K$ is a $\Z_p$-extension, then $\vert H^1(G_n,A(K_\infty)[p^\infty])\vert $ is bounded independently of $n$.
    \item[(b)] If $K_{\infty}/K$ is a $\Z_p\ltimes \Z_p$-extension, then $v_p(\vert H^1(G_n,A(K_\infty)[p^\infty])\vert=O(n)$ and $\dim_{\F_p}(H^1(G_n,A(K_\infty)[p^\infty])[p])=O(1)$.
\end{itemize}
\end{lemma}
\begin{proof}
This proof is completely analogous to the proof of \cite[Proposition 4.1]{debanjana-lim}. We repeat it here for the convenience of the reader. 
Point (a) is a direct consequence of \cite[Lemma 3.4]{limcontrol}.
For point (b) note that $K_\infty$ does not have any non-trivial $p$-extensions (see \cite[Theorem~7.5.3]{nsw}). This implies by the inflation-restriction exact sequence that \[H^1(G_n,A(K_\infty)[p^\infty])\cong H^1(K_n, A[p^\infty])\cong H^1(K_n, A)[p^\infty],\]
where the last isomorphism is due to the fact that $l\neq p$. Using local Tate duality (see \cite[Corollary 3.4]{Mil06}), the last cohomology group is isomorphic to $(A^t(K_n))^\vee[p^\infty]$, where $A^t$ stands for the dual variety. Using again that $l\neq p$, we obtain that it suffices to consider $(A^t(K_n)[p^\infty])^\vee$. This group is clearly finite. As our extension is uniform, the first estimate is immediate from Lemma \ref{growth-condition}. The second estimate in point (b) follows from \cite[lemma 2.1]{debanjana-lim}, since $\dim_{\F_p}(H^1(G_n, \F_p)) = l$ for every $n$ by (the proof of) \cite[Lemma~2.1]{greenberg03}.  
\end{proof}

\begin{lemma}\cite[Proposition 4.3]{debanjana-lim} and \cite[Lemma~3.4]{limcontrol} \label{lemma:Z_p^l}
Let $K$ be a finite extension of $\Q_p$ and assume that $K_\infty/K$ is a $\Z_p^l$-extension. Then 
\[v_p(\vert H^1(G_n, A(K_\infty)[p^\infty])\vert)=O(n).\]
If $l=1$, then we get $O(1)$ instead of $O(n)$.
\end{lemma}

The above sequence of lemmas allows us us to prove the following control theorem for arbitrary abelian varieties and $\Z_p^l$-extensions.

\begin{thm} \label{thm:control1} 
Let $K$ be a number field and let $K_\infty$ be a $\Z_p^l$-extension of $K$.  If $p=2$ we assume that $K$ is totally imaginary. Then the kernel and the cokernel of the natural homomorphism
\[r_n \colon \textup{Sel}_{0,A}(K_n)\to \textup{Sel}_{0,A}(K_\infty)^{G_n}\] 
are finite. Moreover, $v_p(\vert\ker(r_n)\vert)=O(n)$ and $v_p(\vert \coker(r_n)\vert)=O(p^{(l-1)n})$. 
\end{thm} 
\begin{proof} Let $\Sigma$ be a finite set of primes of $K$ which contains both $\Sigma_p$ and $\Sigma_{\textup{br}}(A)$.
Consider the commutative diagram 
\[\begin{tikzcd}[scale cd = 0.85] 
0\arrow[r]&\textup{Sel}_{0,A}(K_n)\arrow[r]\arrow[d,"r_n"]&H^1(K_{\Sigma}/K_n,A[p^\infty])\arrow[r]\arrow[d,"h_n"]& \bigoplus_{v\in \Sigma(K_n)}H^1({K_{n,v}},A[p^\infty])\arrow[d,"g_n"] \\0 \arrow[r]&\textup{Sel}_{0,A}(K_\infty)^{G_n}\arrow[r]&H^1(K_{\Sigma}/K_\infty,A[p^{\infty}])^{G_n}\arrow[r]& \bigoplus_{v\in \Sigma(K_\infty)}H^1({K_{\infty,v}},A[p^{\infty}])\end{tikzcd}
\]
We start by estimating the kernel of $g_n$. By the restriction-inflation exact sequence we see that
\[\ker(g_n)=\bigoplus_{v\in \Sigma(K_n)}H^1(G_{n,v},A(K_{\infty,v})[p^\infty]). \]
If $v$ is totally split in $K_\infty/K$, then $v$ does not contribute to the $\ker(g_n)$. If $v$ lies above a place whose decomposition group has dimension $1$, then the contribution of each term is uniformly bounded independent of $n$ by Lemma \ref{lemma:Z_p^l}. Over each such $v\in \Sigma$ there are $\mathcal{O}(p^{n(l-1)})$ places in $\Sigma(K_n)$ and we obtain in total a contribution of $\mathcal{O}(p^{n(l-1)})$. If the decomposition group has dimension at least $2$, we consider necessarily a place above $p$. By Lemma \ref{lemma:Z_p^l} the contribution is of size $\mathcal{O}(n)$ in this case and there are $\mathcal{O}(p^{n(l-2)})$ primes of $K_n$ above each such prime. In total we obtain that the order of the kernel of $g_n$ is of size $\mathcal{O}(p^{n(l-1)})$. It remains to estimate the kernel and cokernel of $h_n$. Using again the inflation-restriction exact sequence, it suffices to consider $H^1(G_n, A(K_\infty)[p^\infty])$ and $H^2(G_n, A(K_\infty)[p^\infty])$. 
Both groups are finite by Lemma \ref{cohomologyisfinite} and by \cite[Lemma 2.1]{debanjana-lim} the $p$-ranks are bounded independently of $n$.
Note that $A(K_\infty)[p^\infty]^\vee$ is finitely generated over $\Z_p\llbracket G\rrbracket $. Using \cite[Lemma 2.1.1]{Lim-Liang} together with \cite[Proposition 1.9.1]{nsw}\footnote{This lemma is stated only for finite groups, but it can be generalised to pro-p groups easily.}, we can bound the exponent of the torsion subgroup of both cohomology groups by $p^{ln+d}$ for some constant $d$. So we obtain that $v_p(|\ker(h)|)$ and $v_p(|\coker(h)|)$ are of size $\mathcal{O}(n)$. From this both claims are immediate.
\end{proof}

Even though we will apply only the above control theorem in our main results, it is relatively easy to prove a similar result for multi-false Tate extensions. We include this result since it may be of independent interest. 

We say that an extension $K_\infty/K$ is a \emph{multi-false Tate extension} if it is of the form $K_\infty=K(\mu_{p^\infty},\sqrt[p^\infty]{\alpha_1},\dots \sqrt[p^\infty]{\alpha_l})$ for elements $\alpha_i\in K$ and has Galois group $\Z_p\ltimes \Z_p^l$ over $K$. 
In analogy with the work of Kundu and Lim we obtain the following local result. 
\begin{lemma}\cite[Proposition 4.5]{debanjana-lim}
\label{multifalsetata}
Let $K$ be a finite extension of $\Q_p$. Assume that $K_\infty$ is a multi-false Tate extension. Assume that $A$ is of potentially good reduction at all places above $p$. Then 
\[v_p(\vert H^1(G_n, A(K_\infty)[p^\infty])\vert)=O(1).\]
\end{lemma} 
\begin{proof}
By \cite[Theorem 1.1]{kubo-taguchi} we can conclude that $A(K_\infty)[p^\infty]$ is finite. The claim then follows from a generalization of \cite[Proposition 4.2]{debanjana-lim} for abelian varieties.
\end{proof}
As a consequence we obtain the following control theorem.
\begin{thm} \label{thm:control2} 
Let $K$ be a number field, and let $K_\infty/K$ be a multi-false Tate extension of dimension $l$. Let $\Sigma$ be a finite set of primes of $K$ which contains both $\Sigma_p$, $\Sigma_{\textup{br}}(A)$ and the primes $\Sigma_{\textup{ram}}$ that ramify in $K_\infty/K$. If $p=2$ we assume that $K_\infty$ is totally imaginary. Assume that $A$ has potentially good reduction at all primes of $K$ above $p$. Then the kernel and the cokernel of the natural homomorphism
\[r_n \colon \textup{Sel}_{0,A}(K_n)\to \textup{Sel}_{0,A}(K_\infty)^{G_n}\] 
are finite. Moreover, $v_p(\vert\ker(r_n)\vert)=O(1)$ and $v_p(\vert \coker(r_n)\vert)=O(p^{(l-1)n})$. 

If the decomposition group of all primes in $\Sigma$ has dimension at least $2$, then $v_p(\vert \coker(r_n)\vert)=O(n p^{(l-2)n})$.
\end{thm}
\begin{proof} 
The proof is very similar to the proof of Theorem~\ref{thm:control1}, so we will only sketch it here.
Using Imai's theorem and the generalisation due to Kubo and Taguchi \cite[Theorem 1.1]{kubo-taguchi} we see that $A(K_{\infty,v})[p^\infty]$ is finite for all places $v$ above $p$ in $K_\infty$. This implies that also $A(K_\infty)[p^\infty]$ is finite. So we can use \cite[Lemma 2.2]{debanjana-lim} to conclude that the orders of the kernel and cokernel of the maps $h_n$ (defined as in the proof of Theorem~\ref{thm:control1}) are uniformly bounded. It remains to estimate the kernel of $g_n$. Again the cardinality of the kernel of $g_n$ is bounded by $|\bigoplus_{v}H^1(G_{n,v},A(K_{\infty,v})[p^\infty])|$. If $v$ is a place not dividing $p$ and if the decomposition group has dimension $2$, then $|H^1(G_{n,v},A(K_{\infty,v})[p^\infty])|$ is of size $\mathcal{O}(n)$ according to Lemma \ref{ldifferentp}. If the decomposition group has dimension one then the size of $H^1(G_{n,v},A(K_{\infty,v})[p^\infty])$ is bounded by $\mathcal{O}(1)$. 
It remains to estimate the contribution of the primes above $p$. Using Lemma \ref{multifalsetata} we see that the contribution is $\mathcal{O}(p^{n(l-1)})$ if the decomposition group has dimension one and $\mathcal{O}(p^{n(l-2)})$ if the decomposition group has dimension at least $2$. 
\end{proof}

\section{Main theorems} \label{section:main} 
In this section, we prove our main results comparing the Iwasawa modules of ideal class groups and fine Selmer groups. We start with several results for very general uniform $p$-extensions. In the subsequent subsection, we restrict to $\Z_p^l$-extensions. 
\subsection{Uniform extensions} 
In order to prove our first comparison theorem, we need the following 
\begin{thm}\label{newrankthm}
  Let $A$ be an abelian variety of dimension $d$ defined over $K$, and let $\Sigma$ be a finite set of primes of $K$ containing $\Sigma_p$ and $\Sigma_{\textup{br}}(A)$. If $p = 2$, then we assume that $K$ is totally imaginary. Let $L/K$ be a finite extension which is contained in $K_\Sigma$ and assume that $A[p^k]\subseteq A(L)$. Then 
  \[\vert r_p^k(\textup{Sel}_{0,A}(L))-2dr_p^k(\textup{Cl}(L))\vert\le 2dk(1+2 |\Sigma(L)|) . \] 
\end{thm} 
Note that only the right hand side depends on $\Sigma$. So we can improve this bound by choosing $\Sigma$ as minimal as possible. The dependence on $\Sigma$ on the right hand side comes from the fact that in the proof, we work with the group $\textup{Sel}_{0,A[p^k],\Sigma}$, which really depends on $\Sigma$.
\begin{proof}
For $k=1$ and $L$ being contained in a $\Z_p$-extension of $K$ this is \cite[Theorem 5.1]{lim-murty}. 
As $A[p^k]\subseteq A(L)$ we obtain that $$H^1(K_\Sigma/L,A[p^k])=\textup{Hom}(K_\Sigma/L,A[p^k])$$ as well as $H^1(L_{v},A[p^k])=\textup{Hom}(L_{v},A[p^k])$ for each prime $v$. It follows that
\[ \Sel_{0,A[p^k], \Sigma}(L)\cong \textup{Hom}(\textup{Cl}_\Sigma(L), A[p^k]). \]
As $r_p^k(\textup{Hom}(\textup{Cl}_\Sigma, A[p^k])=2dr_p^k(\textup{Cl}_\Sigma(L))$, the claim follows from Lemmas~\ref{lemma:vergleich_sigma_cl} and \ref{lemma:lim-murty}.
\end{proof}

In the following we use the above results to compare the rank and $\mu$-invariant of the fine Selmer group of an abelian variety over an admissible $p$-adic Lie extension $K_\infty$ of $K$ to the (rank and) $\mu$-invariant of the Iwasawa module arising from class groups. Let $G = \Gal(K_\infty/K)$. 
For an abelian variety $A$ defined over $K$, we 
define $r_{A}^{(K_\infty)} := \rg_{\Z_p\llbracket G\rrbracket }(Y^{(K_{\infty})}_{A})$. Let $E_{A} = \bigoplus_{i=1}^{s_{A}} \Z_p\llbracket G\rrbracket /(p^{e_i^{A}})$ be the elementary $\Z_p\llbracket G\rrbracket $-module associated to $Y_{A}^{(K_{\infty})}[p^{\infty}]$ via \cite[Theorem~3.40]{venjakob} and define  $$f_j^{A} =  \vert \{i\mid e_i^{A}>j\}\vert,$$ $j \in \N$. 

As $Y(K_\infty) = \varprojlim_n Y(K_n)$ is a Noetherian $\Z_p\llbracket G\rrbracket $-module, we can also define the invariant $f_j$ with respect to $Y(K_\infty)$: let $\bigoplus_{i=1}^s \Z_p\llbracket G\rrbracket /(p^{e_i})$ be the elementary $\Z_p\llbracket G\rrbracket $-module associated to $Y(K_\infty)[p^\infty]$, and write $$f_j = \vert \{i \mid e_i > j \} \vert,$$ $j \in \N$. Moreover,  we let $r^{(K_\infty)}$ be the $\Z_p\llbracket G\rrbracket $-rank of $Y(K_\infty)$. 

The following result restates Theorem~\ref{thm:A} from the \nameref{section:1}. 
\begin{thm}
\label{thm:comparing ranks}
  Let $A$ be an abelian variety of dimension $d$. Let $\Sigma$ be a finite set of primes containing $\Sigma_p$ and $\Sigma_{\textup{br}}(A)$. If $p = 2$, then we assume that $K$ is totally imaginary. Let $K_{\infty}/K$ be a uniform $p$-extension of dimension $l$ with Galois group $G$ which is $\Sigma$-ramified, i.e. contained in $K_\Sigma$, and such that none of the primes in $\Sigma$ is completely split in $K_{\infty}/K$. 
   Assume that ${A[p^k]\subseteq A(K_\infty)}$. Then \begin{equation}
   \label{eq:comparing ranks}
       k \cdot r_{A}^{(K_\infty)}+\sum_{i=0}^{k-1}f_i^{A}= 2d( \sum_{i=0}^{k-1} f_i).
   \end{equation} 
\end{thm}
\begin{proof}
Since no prime of $\Sigma$ is completely split in $K_\infty/K$, the $\Z_p\llbracket G\rrbracket $-module $Y(K_\infty)$ is in fact torsion, i.e. $r^{(K_\infty)} = 0$, by \cite[Theorem~6.1]{ochi-venjakob}. Without loss of generality we can assume that $A[p^k]
\subseteq A(K)$. 
Consider the natural map 
\[Y(K_\infty)_{G_n}/p^k Y(K_\infty)_{G_n} \longrightarrow Y(K_n)/p^k Y(K_n).\]  Let $Z_n$ and $W_n$ be the kernel and cokernel of this homomorphism. 
We will show that the $p$-valuation of the order of $Z_n$ and $W_n$ is $\Ok(p^{n(l-1)})$. 
 For each $m\ge n$ we consider the homomorphism
 \[Y(K_m)_{G_{m,n}}/p^k Y(K_m)_{G_{m,n}} \longrightarrow Y(K_n)/p^kY(K_n), \] 
 where $G_{m,n} = \Gal(K_m/K_n)$. 
 If $Z_{m,n}$ and $W_{m,n}$ denote the kernel and cokernel of this map, then ${Z_n = \varprojlim_m Z_{m,n}}$ and ${W_n = \varprojlim_m W_{m,n}}$. 
 
 Now fix $m \ge n$, and let $L$ be the maximal abelian extension of $K_n$ that is unramified over $K_m$. Then $$Y(K_m)_{G_{m,n}} \cong \Gal(L/K_m).$$ 
 Let $L'$ be the maximal unramified abelian extension of $K_n$. Then we have the following diagram of fields (where \emph{ur} means an \emph{unramified} extension): 
 \[ \xymatrix{K_m \ar@{-}^{ur}[r] & L \ar@{-}[d] \\ K_m \cap L' \ar@{-}[u] \ar@{-}[r] & L' \\ K_n \ar@{-}[ur]^{ur} \ar@{-}[u] & } \] 
 The kernel and the cokernel of the natural map 
 \[f_{m,n} \colon Y(K_m)_{G_{m,n}} \longrightarrow Y(K_n) \]
 are isomorphic to $\Gal(L/L'K_m)$ and $\Gal((L' \cap K_m)/K_n)$. The $p$-rank of the latter group is bounded independently of $n$ and $m$; since the exponent of the induced cokernel $W_{m,n}$ is at most $k$ and as the upper bounds for both the rank and the exponent do not depend on $n$ or $m$, we have shown that $v_p(|W_{m,n}|) = \Ok(1)$ and $v_p(|W_n|) = \Ok(1)$. 
 
 Moreover, if $H_{m,n}$ denotes the subgroup of $\Gal(L/K_n)$ generated by the inertia subgroups of the primes in $\Sigma(K_n)$, then $L' = L^{H_{m,n}}$. 
 As $L/K_m$ is unramified, we see that the inertia group of each prime in $\Sigma(K_n)$ is isomorphic to a subgroup of $\Gal(K_m/K_n)$, i.e. the $p$-rank of each inertia subgroup is bounded uniformly in $m$ and $n$. In particular, ${\rg_p(H_{m,n}) = \Ok(|\Sigma(K_n)|)}$ and 
 $$\tilde{r}^k_p(|\Gal(L/L'K_m)|) = \Ok(l\vert \Sigma(K_n)\vert) = \Ok(|\Sigma(K_n)|). $$  
 Hence, as no prime in $\Sigma$ is totally split in $K_{\infty}/K$, the $p$-valuation of this kernel is of size $\Ok(p^{n(l-1)})$ in view of \cite[Lemma~4.1]{limwildkernel}. This upper bound does not depend on $m$ and we can conclude that $v_p(|Z_n|)=\Ok(p^{n(l-1)})$ because 
 $$v_p(|Z_{m,n}|) \le \tilde{r}_p^k(\ker(f_{m,n})) + \tilde{r}_p^k(\coker(f_{m,n}))$$ 
 for each $m,n \in \N$, as $\vert \coker(f_{m,n})/p^k\coker(f_{m,n})\vert $ is a natural upper bound for $\vert Z_{m,n}/\pi(\ker(f_{m,n}))\vert$ (as in the proof of Lemma~\ref{lemma:ckmm1}). Here 
 $$\pi:Y(K_m)_{G_{m,n}} \longrightarrow Y(K_m)_{G_{m,n}}/p^k Y(K_m)_{G_{m,n}} $$ 
 denotes the canonical surjection.  

Therefore $r^k_p(Y(K_n)) = \tilde{r}_p^k(Y(K_n))$ satisfies \begin{eqnarray} \label{eq:6.4} 
 |r^k_p(Y(K_n)) - v_p(|Y(K_\infty)_{G_n}/ p^k Y(K_\infty)_{G_n}|)| = \Ok(p^{n(l-1)}).\end{eqnarray} 
Let $E$ be the elementary module associated to $Y(K_\infty)[p^\infty]$. Then $E$ is of the form 
\[\bigoplus_{i=1}^s\Z_p\llbracket G\rrbracket /(p^{e_i}). \]
by \cite[Theorem~3.40]{venjakob}. It follows from Theorem~\ref{thm:perbet+} that 
\[ |v_p(|Y(K_\infty)_{G_n}/p^k Y(K_\infty)_{G_n}|) - \sum_{i=0}^{k-1} f_i \cdot p^{nl}| = \Ok(k p^{n(l-1)}).\] 
Using Theorem \ref{newrankthm} and the fact that $\vert \Sigma(K_n)\vert =\Ok(p^{n(l-1)})$, we obtain
$$\bigl| r^k_p(\textup{Sel}_{0,A}(K_n))-2 d (\sum_{i=0}^{k-1}f_ip^{nl})\bigl |=\Ok(p^{n(l-1)}).$$
\begin{lemma} \label{lemma:katharinas_control_thm} 
  Let $K_\infty/K$ be a uniform $\Sigma$-ramified $p$-extension with Galois group $G$ of dimension $l$, and let $A$ be an abelian variety defined over $K$. Assume that no prime in $\Sigma$ is totally split in $K_\infty/K$. Then 
  $$\vert r^k_p(\textup{Sel}_{0,A}(K_{\infty})^{G_n})-r^k_p(\textup{Sel}_{0,A}(K_{n}))\vert=\Ok(k p^{n(l-1)}),$$ 
  where the implicit constant does not depend on $n$ or $k$. 
\end{lemma} 
\begin{proof} 
  It follows from \cite[Lemma~2.1]{greenberg03} that $v_p(|H^1(G_n, \Z/p\Z|) = l$ for every ${n \in \N}$, and that $ v_p(|H^2(G_n, \Z/p\Z)|) = \mathcal{O}(1)$. Therefore the $p$-rank of 
  $$H^i(G_n, A(K_{\infty})[p^\infty])[p]$$ for $1\le i \le 2$ is uniformly bounded by \cite[Lemma 2.1]{debanjana-lim}. Consider
  \[\begin{tikzcd}[scale cd = 0.85] 
  0\arrow[r]&\textup{Sel}_{0,A}(K_n)\arrow[r]\arrow[d,"r_n"]&H^1(K_{\Sigma}/K_n,A[p^\infty])\arrow[r,"\rho"]\arrow[d,"h_n"]& \bigoplus_{v\in \Sigma(K_n)}H^1({K_{n,v}},A[p^\infty])\arrow[d,"g_n"] \\0 \arrow[r]&\textup{Sel}_{0,A}(K_\infty)^{G_n}\arrow[r]&H^1(K_{\Sigma}/K_\infty,A[p^{\infty}])^{G_n}\arrow[r]& \bigoplus_{v\in \Sigma(K_\infty)}H^1({K_{\infty,v}},A[p^{\infty}])\end{tikzcd}
  \]
 By Lemma \ref{lemma-ranks}, it suffices to compute $r_p^k(\ker(r_n))$ and $r_p^k(\coker(r_n))$ to estimate $$\vert r^k_p(\textup{Sel}_{0,A}(K_{\infty})^{G_n})-r^k_p(\textup{Sel}_{0,A}(K_{n}))\vert.$$ 
  As $r_p^k(\ker(r_n))\le r_p^k(\ker(h_n))=r_p^k(H^1(G_n,A(K_\infty)[p^\infty])=\mathcal{O}(k)$, where the implicit constant does not depend on $k$ or $n$, it remains to consider $\coker(r_n)$. Consider the sequence
  \[0\to W/\rho(\ker(h_n))\to \coker(r_n)\to\coker(h_n),\]
  where $W\subseteq \ker(g_n)$. Again using Lemma \ref{lemma-ranks} we see that $$\vert r_p^k(W/\Image(\ker(h_n))-r_p^k(\coker(r_n))\vert \le r_p^k(\coker(h_n))=\mathcal{O}(k).$$ 
  As $r_p^k(\ker(g_n))=\mathcal{O}(k p^{(l-1)n})$ ($\mathcal{O}(k)$ for every decomposition group of dimension at least $1$), we obtain indeed
  $\vert r^k_p(\textup{Sel}_{0,A}(K_{\infty})^{G_n})-r^k_p(\textup{Sel}_{0,A}(K_{n}))\vert=\mathcal{O}(p^{n(l-1)})$
\end{proof} 
Therefore  
$$\Bigl| r^k_p(\textup{Sel}_{0,A}(K_{\infty})^{G_n})-2d\left(\sum_{i=0}^{k-1}f_ip^{nl}\right)\Bigr|=\Ok(p^{n(l-1)}).$$
In particular, 
$$\Bigl| \tilde{r}_p^k((Y^{(K_{\infty})}_{A})_{G_n}-2d\left(\sum_{i=0}^{k-1}f_ip^{nl}\right)\Bigr|=\Ok(p^{n(l-1)}). $$  Applying Theorem~\ref{thm:perbet+} to $Y^{(K_{\infty})}_{A}$ yields 
$$\Bigl| \tilde{r}_p^k((Y^{(K_{\infty})}_{A})_{G_n} -\sum_{i=0}^{k-1}f_i^{A}p^{nl}-kp^{nl}r_{A}^{(K_\infty)} \Bigr| =\Ok(p^{n(l-1)}).$$
Comparing the last two equations, we obtain the claim of the theorem.
\end{proof}

We mention several consequences of our first comparison result. 
\begin{cor} \label{cor:weak-Leo}
  Let $A/K$ and $\Sigma$ be as in Theorem~\ref{thm:comparing ranks}. Let $K_\infty/K$ be a uniform $\Sigma$-ramified $p$-extension of dimension $l$, $l \ge 1$, with Galois group $G$. We assume that none of the primes in $\Sigma$ is completely split in $K_\infty/K$. 
  
  Suppose that $A[p^k] \subseteq A(K_\infty)$ for some $k > \max_i e_i$, where $\bigoplus_{i=1}^s \Z_p\llbracket G\rrbracket /(p^{e_i})$ is the elementary module attached to $Y(K_\infty)$. Then $r_{A}^{(K_\infty)} = 0$ and $f_i^{A}=2df_i$ for all $i$.
\end{cor} 
\begin{proof} 
  Applying Theorem~\ref{thm:comparing ranks} with $k$ and $k-1$ yields the same expression on the right hand side of Equation \eqref{eq:comparing ranks}. Thus, $r_{A}^{(K_\infty)} = f_k^{A} = 0$. If we now apply  Theorem~\ref{thm:comparing ranks} for $1\le i\le k-1$ successively, we obtain $f_i^A=2df_i$ for $1\le i\le k-1$. 
\end{proof} 
\begin{rem}
  For example, the condition on $A[p^k]\subseteq A(K_\infty)$ is automatically satisfied for the trivialising extension $K_\infty=K(A[p^\infty])$.
\end{rem}

The property $r_{A}^{(K_\infty)} = 0$, i.e. the module $Y_{A}^{(K_\infty)}$ being torsion as a $\Z_p\llbracket G\rrbracket $-module, is often referred to as the \emph{weak Leopoldt conjecture for $A$ over $K_\infty$} (see also \cite[Lemma~7.1]{Lim17} for equivalent formulations of this conjecture). 
\begin{rem}
It should not be a surprise that the weak Leopoldt conjecture holds in the situation described in Corollary~\ref{cor:weak-Leo}. In fact, Coates and Sujatha have proved in \cite[Lemmas~2.4 and 3.1]{coates-sujatha} 
that $r_{A}^{(K_\infty)} = 0$ holds for an elliptic curve $A = E$ whenever $K_\infty$ is admissible (in particular, it should contain the cyclotomic $\Z_p$-extension of $K$) and $A(K_\infty)$ contains $A[p^\infty]$. For general abelian varieties, Ochi and Venjakob have shown in (the proof of) \cite[Corollary~4.8]{ochi-venjakob2} that $r_{A}^{(K_\infty)} = 0$ for $K_\infty = K(A[p^\infty])$. Note that the cyclotomic $\Z_p$-extension of $K$ is automatically contained in $K_\infty$ if $A[p^\infty] \subseteq A(K_\infty)$ (this follows from the Weil pairing, as explained in the proof of \cite[Corollary~4.8]{ochi-venjakob2}). 

We can prove that $r_{A}^{(K_\infty)}=0$ holds in slightly more generality, since we only need $A[p^k]\subseteq A(K_{\infty})$. Moreover, the decomposition hypothesis will hold if $K_\infty$ contains the cyclotomic $\Z_p$-extension of $K$, but this latter condition is not necessary.  
\end{rem}

On the other hand, if the weak Leopoldt conjecture \emph{fails} for $A$ over some uniform $p$-extension, then this would have certain interesting consequences: as a further application of Theorem~\ref{thm:comparing ranks}, we describe a setting where non-trivial (and in fact arbitrarily large) $\mu$-invariants might occur. The main hypothesis is related to the failure of the weak Leopoldt conjecture (see also Remark~\ref{rem:leo} below). 
\begin{cor} \label{cor:muunbounded} 
Let $A$ be an abelian variety of dimension $d$ defined over $K$ and let $\Sigma$ be a finite set of primes containing $\Sigma_{\textup{br}}(A)$ and $\Sigma_p$. If $p=2$ then we assume that $K$ is totally imaginary. Let $K_{\infty}/K$ be a uniform $\Sigma$-ramified $p$-extension such that none of the primes in $\Sigma$ is completely split in $K_{\infty}/K$. Assume that $A[p]\subseteq A(K_\infty)$ and that ${r := \rg_{\Z_p\llbracket G\rrbracket }(Y_{A}^{(K_\infty)}) \ge 1}$. Let $L^{(n)}=K(A[p^n])$ and $L_{\infty}^{(n)}=K_{\infty}L^{(n)}$. Then ${\mu(L_{\infty}^{(n)}/L^{(n)})}$ becomes arbitrarily large.
\end{cor}
\begin{proof}
By hypothesis we have $A[p^n]\subseteq A(L^{(n)})$ for each $n \in \N$. Let $M'$ be any number field such that $L^{(n)}\subseteq M'\subseteq L^{(n)}_{\infty}$ and let $M=L^{(n-1)}_{\infty}\cap M'$. Then $$|\Gal(M'/M)|\le |\Gal(L_\infty^{(n)}/L_\infty^{(n-1)})|. $$ 
In particular, the order of the kernel of the restriction map $\textup{Sel}_{0,A}(M)\longrightarrow\textup{Sel}_{0,A}(M')$ is uniformly bounded as one runs over the number fields $M' \subseteq L_\infty^{(n)}$, and therefore the dualised map  $Y_{A}^{L^{(n)}_{\infty}}\longrightarrow Y_{A}^{L^{(n-1)}_{\infty}}$ has a uniformly bounded cokernel. Hence, $\rg_{\Z_p\llbracket G\rrbracket }(Y_{A}^{(L_\infty^{(n)})}) \ge r$ for every $n \in \N$. Theorem \ref{thm:comparing ranks} and the definition of the invariants $f_i^{A}$ imply that 
\[ 2d \mu(L_{\infty}^{(n)}/L^{(n)})\ge 2d\sum_{i=0}^{n-1}f_i^{A} + nr \ge nr.\]
\end{proof}
\begin{rem} \label{rem:leo} 
  It seems remarkable that we obtain lower bounds for the $\mu$-invariants of ideal class groups from a hypothesis on the fine Selmer groups. In the literature, we have seen so far only arguments deriving lower bounds for (fine) Selmer groups from information about ideal class groups, as in \cite{kundu-mu}. 
  
  Note that the $\Z_p\llbracket G \rrbracket $-rank $r$ of the fine Selmer group is larger than 0 if and only if the \emph{weak Leopoldt conjecture} of $A$ over $K_\infty$ fails (see \cite[Lemma~7.1]{Lim17}). No example is known where this happens. 
\end{rem} 

Finally, Theorem~\ref{thm:comparing ranks} can be used in order to compare the Selmer groups of two different abelian varieties in suitable settings. For example, the following corollary applies to congruent abelian varieties in the sense of \cite{KM21}. 
\begin{cor} 
  Let $A_1$ and $A_2$ be two abelian varieties defined over $K$, let $\Sigma$ be a finite set of primes containing $\Sigma_p$ and both $\Sigma_{\textup{br}}(A_i)$, $i = 1,2$. If $p = 2$, then we assume that $K$ is totally imaginary. Suppose that $K_\infty/K$ is a uniform $\Sigma$-ramified $p$-extension such that no prime of $\Sigma$ is completely split in $K_\infty/K$.
  
  Assume further that $A_1[p^k] \subseteq A_1(K_\infty)$ and $A_2[p^k] \subseteq A_2(K_\infty)$. Then 
 $$vr_{A_1}^{(K_\infty)}+\sum_{i=0}^{v-1}f_i^{A_1}=  vr_{A_2}^{(K_\infty)}+\sum_{i=0}^{v-1}f_i^{A_2}$$ for all $v\le k$. If in addition $r_{A_1}^{(K_\infty)}=r_{A_2}^{(K_\infty)}$, then we can deduce that $f_i^{A_1}=f_i^{A_2}$ for all $i
 \le k-1$. 
\end{cor} 

Now we prove a weak version of Theorem~\ref{thm:comparing ranks} for number fields $L$ which satisfy a weaker hypothesis than ${A[p^k] \subseteq A(L)}$, using the following lemma which generalises a result of Lim and Murty (see \cite[Lemma~4.3]{lim-murty}). Let $\Sigma$ be a finite set of primes of $K$ which contains the primes above $p$ and the primes of bad reduction of $A$. 
\begin{lemma} \label{lemma:murty_ungl} 
  Let $A$ be an abelian variety of dimension $d$, defined over a number field $L \supseteq K$, and let $\Sigma$ be as above. Then 
  $$ \rg^k_p(\Sel_{0,A}(L)) \ge \rg^k_p(\textup{Cl}_\Sigma(L)) \cdot \rg_p(p^{k-1}A(L)[p^\infty]) - 2kd.$$ 
\end{lemma} 
Of course this inequality is non-trivial only if ${A(L)[p^k]\setminus A(L)[p^{k-1}] \ne \{0\}}$, but this condition is much weaker than $A[p^k] \subseteq A(L)$. For example, we could consider an abelian variety of large dimension which has complex multiplication. We have seen in Corollary~\ref{cor:pntorsion} that there exists a multiple $\Z_p$-extension $K_\infty'$ over a potentially slightly larger field $K'$ which contains $A[p^\infty]$, and that we may in fact assume that $A[p^n] \subseteq A(K_n')$ for every $n \in \N$. But then it is possible to realise an infinite submodule of $A[p^\infty]$ over a multiple $\Z_p$-extension of relatively small dimension. For example, it is possible to choose a single $\Z_p$-extension of $K'$ such that the $n$-th layer of this $\Z_p$-extension contains some point in $A[p^n] \setminus A[p^{n-1}]$ for each $n \in \N$. 

In order to prove the above inequality we need the following auxiliary result. 
\begin{lemma}
Let $G$ be a pro-p group and $M$ a discrete $G$-module which is cofinitely generated over $\Z_p$. 
Let $h_k(G)=v_p(\vert H^1(G,\Z/p^k\Z)\vert)$. Then we have the following inequality:
\[h_k(G)r_p(p^{k-1}M^G)-r_p^k((M/M^G)^G)\le r_p^k(H^1(G,M)).\]
\end{lemma}
\begin{proof}
   We consider first a module $M$  with trivial $G$-action and want to prove that $r_p^k(H^1(G,M))\ge h_k(G)r_p(p^{k-1}M)$. As $H^1(G,M)=\textup{Hom}(G,M)$, it suffices to consider the groups $H^1(G,\Z/p^r\Z)$ for $r\ge 1 $ and $H^1(G, \Q_p/\Z_p)$. For $\alpha\ge 0$ we have the natural inclusion 
   \[H^1(G,\Z/p^k\Z)\hookrightarrow H^1(G,\Z/p^{k+\alpha}\Z)\hookrightarrow H^1(G,\Q_p/\Z_p).\]
   Using the fact that $\Q_p/\Z_p$ is divisible and that $H^0(G,\Q_p/\Z_p)=\Q_p/\Z_p$, we obtain that ${H^1(G,\Q_p/\Z_p)[p^k]\cong H^1(G,(\Q_p/\Z_p)[p^k]) \cong H^1(G, \Z/p^k\Z)}$. Therefore, $h_k(G)=r_p^k(H^1(G,\Q_p/\Z_p))$. Note that the number of direct summands of $M$ that are not annihilated by $p^{k-1}$ is given by $r_p(p^{k-1}(M))$ and each of these terms contributes $h_k(G)$ to $r_p^k(H^1(G,M))$.  
   
   For the rest of the proof we consider an arbitrary $G$-module. Taking the cohomology of the exact sequence
   \[0\longrightarrow M^G\longrightarrow M\longrightarrow M/M^G\longrightarrow 0,\]
   we obtain a sequence
   \[(M/M^G)^G\longrightarrow H^1(G,M^G)\longrightarrow H^1(G,M).\]
   This in turn implies
   \begin{eqnarray*} r_p^k(H^1(G,M)) & \ge & r_p^k(H^1(G,M^G))-r_p^k((M/M^G)^G) \\ 
   & \ge & h_k(G)r_p(p^{k-1}M^G)-r_p^k((M/M^G)^G). 
   \end{eqnarray*} 
\end{proof} 
\begin{proof}[Proof of Lemma~\ref{lemma:murty_ungl}]
 Let $H_\Sigma = H_{\Sigma(L)}(L)$ be the maximal unramified $p$-extension of $L$ in which each prime of $L$ dividing some $v \in \Sigma$ is completely split. Consider the following diagram with exact rows.
 \[\begin{tikzcd}[scale cd = 0.95] 0\arrow[r]&\textup{Sel}_{0,A}(L)\arrow[r]\arrow[d,"s"]&H^1(L_{\Sigma}/L,A[p^\infty])\arrow[r]\arrow[d,"h"]& \bigoplus_{v\in \Sigma(L)}H^1(L_{v},A[p^\infty])\arrow[d,"g"] \\ 
 0 \arrow[r]&\textup{Sel}_{0,A}(H_\Sigma)\arrow[r]&H^1(L_{\Sigma}/H_{\Sigma},A[p^{\infty}])\arrow[r]& \bigoplus_{w\in \Sigma(H_\Sigma)}H^1({H_{\Sigma,w}},A[p^{\infty}]).\end{tikzcd}\]
 Here $H_{\Sigma,w}$ denotes the completion of $H_\Sigma$ at ${w \in \Sigma(H_\Sigma)}$. As the decomposition group $G_v$ of every prime in $\Sigma$ in the extension $H_\Sigma/L$ is trivial, we obtain that $H^1(G_v, A(H_{\Sigma,w})[p^{\infty}])$ is trivial. On the other hand, by the inflation-restriction exact sequence we know that $\bigoplus H^1(G_v, A(H_{\Sigma,w})[p^\infty])$ is the kernel of the third vertical map. Hence, we see that the kernels of the first two vertical maps are isomorphic. Using again the inflation-restriction sequence yields an injection of $H^1(H_\Sigma/L,A(H_\Sigma)[p^\infty])$ into $\textup{Sel}_{0,A}(L)$. Thus, we have
 \begin{align*}
     r_p^k(\textup{Sel}_{0,A}(L))&\ge r_p^k(H^1(H_\Sigma/L, A(H_\Sigma)[p^\infty]))\\
     &\ge h_k(\Gal(H_\Sigma/L))r_p(p^{k-1}A(L)[p^\infty]) \\ &\quad -r_p^k((A(H_\Sigma)[p^\infty]/A(L)[p^\infty])^{\Gal(H_\Sigma/L)})\\
     &\ge h_k(\Gal(H_\Sigma/L))r_p(p^{k-1}A(L)[p^\infty])- 2dk.
 \end{align*}
 As the group $H_\Sigma/L$ is finite, we see that $$h_k(H_\Sigma/L)=r_p^k(\textup{Hom}(\textup{Cl}_\Sigma(L), \Z/p^k\Z))=r_p^k(\textup{Cl}_\Sigma(L)). $$ 
\end{proof}
Now we can prove our second comparison theorem, which is a variant of Theorem~\ref{thm:comparing ranks}. In this result, instead of assuming that $A[p^k] \subseteq A(K_\infty)$, we only need that ${A(K_\infty)[p^k] \setminus A(K_\infty)[p^{k-1}]}$ is non-trivial. 
\begin{thm} \label{thm:muungleichung} 
  Let $A$ be an abelian variety of dimension $d$ defined over $K$.  If $p = 2$, then we assume that $K$ is totally imaginary. Let $\Sigma$ be a finite set of primes containing $\Sigma_p$ and $\Sigma_{\textup{br}}(A)$. Let $K_{\infty}/K$ be a uniform $\Sigma$-ramified $p$-extension of dimension $l$ with Galois group $G$ such that none of the primes in $\Sigma$ is completely split in $K_{\infty}/K$. 
  
  We suppose that $p^{k-1}A(K_\infty)[p^\infty] \ne \{0\}$. 
  Then $$r_{A}^{(K_\infty)} + \sum_{i=0}^{k-1}f_i^{A} \ge \sum_{i=0}^{k-1}f_i. $$ 
  In particular, if $\mu(Y(K_\infty)) > 0$, then either $r_{A}^{(K_\infty)} > 0$ or $\mu(Y_{A}(K_\infty)) > 0$. 
\end{thm} 
\begin{proof} 
  Lemma~\ref{lemma:murty_ungl} implies that 
  \[ r^k_p(\Sel_{0,A}(K_n)) \ge r_p^k(\textup{Cl}_\Sigma(K_n)) - 2kd \] 
  for each sufficiently large $n \in \N$. Since $|\Sigma(K_n)| = \Ok(p^{(l-1)n})$, Lemma~\ref{lemma:vergleich_sigma_cl} implies that there exists a constant $C \in \N$ such that 
  \[ r_p^k(\Sel_{0,A}(K_n)) \ge r_p^k(Y(K_n)) - C \cdot p^{(l-1)n}\] 
  for each $n \in \N$. Now we deduce from \eqref{eq:6.4} and Theorem~\ref{thm:perbet+} for $k$ that 
  \begin{eqnarray*} 
     & & |r_p^k(Y_n) - v_p(|Y(K_\infty)_{G_n}/p^k Y(K_\infty)_{G_n}|)| = \Ok(k p^{(l-1)n}) \\ 
     & \Longrightarrow & |r_p^k(Y_n) - (r_{\Z_p \llbracket G \rrbracket }(Y(K_\infty)) \cdot k + \mu^{(k)}(Y(K_\infty))) \cdot p^{nl}| =  \Ok(kp^{(l-1)n}) \\ 
     & \Longrightarrow & |r_p^k(Y_n) - \sum_{i=0}^{k-1}f_i \cdot p^{nl}| = \Ok(p^{(l-1)n}), 
  \end{eqnarray*} 
  since $r_{\Z_p\llbracket G \rrbracket }(Y(K_\infty)) = 0$ by our decomposition hypothesis and $k$ is fixed. 
  
  On the other hand, the control theorem used in the proof of Theorem~\ref{thm:comparing ranks} (see Lemma~\ref{lemma:katharinas_control_thm}) implies that 
  \[ |r_p^k(\Sel_{0,A}(K_n)) - r_p^k((Y_{A}^{(K_\infty)})_{G_n})| = \Ok(p^{k(l-1)n}). \] 
  The theorem therefore follows via a second application of Theorem~\ref{thm:perbet+}, this time for the Iwasawa module $Y_{A}^{(K_\infty)}$. 
\end{proof} 

\begin{cor} 
  Let $A$ be an abelian variety with complex multiplication defined over a number field $K$. We assume that there exists a $\Z_p^l$-extension of $K$ which contains $A[p^\infty]$ (cf. Corollary~\ref{cor:pntorsion}). Let $\Sigma$ be a finite set of primes of $K$ containing $\Sigma_p$ and $\Sigma_{\textup{br}}(A)$. If $p = 2$, then we assume $K$ to be totally imaginary. 
  
  If $K_\infty$ denotes any $\Sigma$-ramified uniform $p$-extension of $K$ of dimension $l$ which contains the cyclotomic $\Z_p$-extension, then 
  \[ \rg_{\Z_p\llbracket G\rrbracket }(Y_{A}^{K_\infty)}) + \mu(Y_{A}^{(K_\infty)}) \ge \mu(Y(K_\infty)). \]  
\end{cor} 
\begin{proof} 
  This follows from Theorem~\ref{thm:muungleichung} by noting that the cyclotomic $\Z_p$-extension $K_\infty^c$ of $K$ is contained in $K(A[p^\infty])$ (cf. the proof of \cite[Corollary~4.8]{ochi-venjakob2}), and that no prime of $K$ is totally split in $K_\infty^c/K$. Since $K_\infty^c$ contains $p$-power torsion points of arbitrarily high order, we can take $k$ sufficiently large in Theorem~\ref{thm:muungleichung}. 
\end{proof} 

In our third comparison theorem, which is the last one for general uniform $p$-extensions, we compare finer invariants of the Iwasawa modules $Y(K_\infty)$ and $Y_{A}^{(K_\infty)}$. This result has been announced as Theorem~\ref{thm:B} in the \nameref{section:1}. 
\begin{thm} \label{thm:z_p[[H]]} 
  We assume that ${\Sigma = \Sigma_p \cup \Sigma_{\textup{br}}(A)}$. Let $A$ be an abelian variety of dimension $d$, and let $K_\infty/K$ be a uniform $\Sigma$-ramified $p$-adic Lie extension with Galois group $G$ of dimension $l$. We suppose that there exists a uniform subgroup $H$ of $G$ of dimension $h$ such that $Y_{A}^{(K_\infty)}$ is finitely generated over $\Z_p\llbracket H\rrbracket $. Let $F_n = K_\infty^{H_n}$, $n \in \N$, and let $F_\infty = K_\infty^H$. 
  
  If $A[p^\infty] \subseteq A(F_\infty)$ and if all primes of $\Sigma$ are finitely split in $F_\infty/K$,  then $\rg_{\Z_p\llbracket H\rrbracket }(Y(K_\infty))$ is also finite and 
  $$ \rg_{\Z_p\llbracket H\rrbracket }(Y_{A}^{(K_\infty)}) = 2d \cdot \rg_{\Z_p\llbracket H\rrbracket }(Y(K_\infty)). $$ 
\end{thm} 
Note that the fields $F_n$ are infinite extensions of $K$, since they contain the field $F_\infty$. By the  hypotheses, $A[p^\infty] \subseteq F_n$ for each $n \in \N$. Furthermore, we cannot take a larger set $\Sigma$ because of the decomposition hypothesis. 

Generalising from the case of elliptic curves, it seems reasonable to expect that these conditions can be satisfied only by abelian varieties with complex multiplication (see also  \cite[Lemma~2.8]{coates-fragments}).
\begin{proof} 
  The following auxiliary result of Harris will play an important role in the proof. 
  \begin{prop}[Harris] \label{prop:harris} 
  Let $G$ be an $l$-dimensional uniform $p$-group, and let $M$ be a finitely generated $\Z_p\llbracket G\rrbracket $-module. Then 
  $$ |\rg_{\Z_p}(M_{G_n}) - \rg_{\Z_p\llbracket G\rrbracket }(M) p^{ln}| = \mathcal{O}(p^{(l-1)n}). $$ 
\end{prop} 
\begin{proof} 
  This is \cite[Theorem~1.10]{harris}. 
\end{proof} 
  Now we apply the following generalisation of a control theorem of Kundu and Lim. 
  \begin{thm} 
     Let $K_\infty/K$ be a uniform $p$-extension, let $H \subseteq \Gal(K_\infty/K)$ be a uniform subgroup, and let $F_n = K_\infty^{H_n}$, $n \in \N$. Let ${\Sigma = \Sigma_p \cup \Sigma_{\textup{br}}(A)}$, and suppose that all primes of $\Sigma$ are finitely split in $F_\infty/K$. We consider the canonical maps 
     $$ s_n : \Sel_{A,0}(F_n) \longrightarrow \Sel_{A,0}(K_\infty)^{H_n}. $$ 
     Then $\textup{corank}_{\Z_p}(\ker s_n) = \mathcal{O}(1)$ and $\textup{corank}_{\Z_p}(\coker s_n) = \mathcal{O}(p^{n (h - 1)})$. 
  \end{thm} 
  
  \begin{proof} 
     Starting with a commutative diagram 
     \[ \begin{tikzcd}[scale cd = 0.85] 0 \arrow[r] & \Sel_{A,0}(F_n) \arrow[r] \arrow[d] & H^1(G_\Sigma(F_n), A[p^\infty]) \arrow[r] \arrow[d] & \bigoplus_{v_n \in \Sigma(F_n)} H^1(F_{n, v_n}, A[p^\infty]) \arrow[d]\\ 0 \arrow[r] & \Sel_{0,A}(K_\infty)^{H_n} \arrow[r] & H^1(G_\Sigma(K_\infty), A[p^\infty])^{H_n} \arrow[r] & \bigoplus_{w \in \Sigma(K_\infty)} H^1(F_{\infty,w}, A[p^\infty]),\end{tikzcd}\] 
     we can use the same arguments as have been employed in the proof of \cite[Theorems~5.1 and 5.5]{debanjana-lim} (cf. also the proofs of Theorems~\ref{thm:control1} and \ref{thm:control2}). 
  \end{proof} 
  It follows from this control theorem and Harris' Proposition~\ref{prop:harris} that 
  \[ | \rg_{\Z_p}(Y_{A}^{(F_n)}) - \rg_{\Z_p}((Y_{A}^{(K_\infty)})_{H_n})| = \Ok(p^{(h-1)n}) \] 
  and 
  \[ | \rg_{\Z_p\llbracket H\rrbracket }(Y_{A}^{(K_\infty)}) p^{hn} - \rg_{\Z_p}((Y_{A}^{(K_\infty)})_{H_n})| =\Ok(p^{(h-1)n}). \] 
  In particular, $\rg_{\Z_p}(Y_{A}(F_n))$ is finite for all $n \in \N$. 
  
  On the other hand, since $A[p^\infty] \subseteq A(F_n)$ for each $n \in \N$, one can conclude as in the proof of Theorem~6.3 that 
  ${\Sel_{A,0}(F_n) \cong \textup{Hom}(\textup{Cl}_\Sigma(F_n), A[p^\infty]) \cong Y_\Sigma(F_n)^{2d}}$ and therefore $Y_{A}^{(F_n)} \cong (\textup{Cl}_\Sigma(F_n) \otimes \Z_p)^{2d}$ as $\Z_p$-modules for each $n$. It is clear that 
  \[ | \rg_{\Z_p}(Y_\Sigma(F_n)) - \rg_{\Z_p}(Y(F_n))|  = \Ok(|\Sigma(F_n)|) = \Ok(p^{(h-1)n}), \] 
  since each prime $v \in \Sigma$ is finitely split in $F_\infty/K$ by hypothesis. 
  
  Moreover, 
  \[ |\rg_{\Z_p}(Y(F_n)) - \rg_{\Z_p}(Y(K_\infty)_{H_n})| = \Ok(p^{(h-1)n}). \] 
  Indeed, class field theory implies that $Y(K_\infty)_{H_n} \cong \Gal(L_n/K_\infty)$, where $H(K_\infty)$ denotes the maximal abelian unramified pro-$p$-extension of $K_\infty$, as in Section~\ref{subsection:classnumberformulas}, and where ${L_n \subseteq H(K_\infty)}$ denotes the maximal subextension which is abelian over $F_n$. Since ${H(K_\infty)/K_\infty}$ is unramified and $\rg_{\Z_p}(H_n) = \Ok(1)$, the $\Z_p$-rank of the sum of the inertia subgroups of $\Gal(L_n/F_n)$, which generate $\Gal(L_n/H(F_n))$ because $\Sigma$ contains the primes which ramify in $K_\infty/F_n$, is $\Ok(|\Sigma(F_n)|) = \Ok(p^{(h-1)n})$. In particular, $Y(K_\infty)$ is also finitely generated over $\Z_p\llbracket H\rrbracket $. 
  
  Summarising, we have shown that 
  \[ | \rg_{\Z_p\llbracket H\rrbracket }(Y_{A}^{(K_\infty)}) p^{hn} - 2d \cdot \rg_{\Z_p}(Y(K_\infty)_{H_n})| =  \Ok(p^{(h-1)n}) \] 
  and therefore 
  \[ |\rg_{\Z_p\llbracket H\rrbracket }(Y_{A}^{(K_\infty)}) p^{hn} - 2d \cdot \rg_{\Z_p\llbracket H\rrbracket }(Y(K_\infty)) p^{hn}| = \Ok(p^{(h-1)n}), \] 
  where the last equation follows from the proposition~\ref{prop:harris} of Harris. 
\end{proof} 
Note: although $\rg_{\Z_p\llbracket H\rrbracket }(Y(F_\infty)) = 0$ because of our decomposition condition, by \cite[Theorem~6.1]{ochi-venjakob}, it remains possible that $\rg_{\Z_p\llbracket H\rrbracket }(Y(K_\infty)) > 0$.

\subsection{$\Z_p^l$-extensions} 
Now we restrict from general uniform $p$-extensions to multiple $\Z_p$-extensions. In order to formulate the first  result, we recall that the set $\mathcal{E}(K)$ of $\Z_p$-extensions of a number field $K$ bears a natural topology, which has been introduced by Greenberg in \cite{green_73}: a basis of neighbourhoods of $K_\infty \in \mathcal{E}(K)$ is given by the sets 
\[ \mathcal{E}(K_\infty,n) := \{ \tilde{K}_\infty \in \mathcal{E}(K) \mid [(K_\infty \cap \tilde{K}_\infty) : K] \ge p^n\},  \] 
$n \in \N$. For later purpose, we also define the following notation: if $K_\infty/K$ is a $\Z_p^l$-extension, $l \in \N$, then we let $\mathcal{E}^{\subseteq K_\infty}(K)$ be the set of $\Z_p$-extensions of $K$ which are contained in $K_\infty$. 
\begin{thm} \label{thm:weak-Leo_dense} 
  Let $A$ be an abelian variety of dimension $d$ which is defined over $K$. Let $\Sigma$ be a finite set of primes of $K$ which contains $\Sigma_p$ and $\Sigma_{\textup{br}}(A)$. Let $K_\infty/K$ be a $\Z_p^l$-extension, $l \ge 1$, with Galois group $G$. We assume that no prime in $\Sigma$ is completely split in $K_\infty/K$. 
  
  Suppose that $A[p^k] \subseteq A(K)$ for some $k > \max_i e_i$, where $\bigoplus_{i=1}^s \Z_p\llbracket G\rrbracket /(p^{e_i})$ is the elementary module attached to $Y(K_\infty)$. Then the weak Leopoldt conjecture for $A$ holds over a dense subset of $\mathcal{E}^{\subseteq K_\infty}(K)$. 
\end{thm} 
\begin{proof} 
  We may assume that $l \ge 2$; Corollary~\ref{cor:weak-Leo} implies that $r_{A}^{(K_\infty)} = 0$. If $K_\infty/K$ is a $\Z_p^2$-extension, then $r_{A}^{(K_\infty^{(1)})} = 0$ for all but finitely many $\Z_p$-extensions $K_\infty^{(1)}$ of $K$ in $\mathcal{E}^{\subseteq K_\infty}(K)$ (see the proofs of Lemmas~7.8 and 7.9 of \cite{non-torsion}). In fact, if $K_\infty^{(1)} \in \mathcal{E}^{\subseteq K_\infty}(K)$, ${\Gamma = \Gal(K_\infty/K_\infty^{(1)}) \cong \Z_p}$, $\Lambda_1 = \Z_p\llbracket \Gal(K_\infty^{(1)}/K)\rrbracket $ and ${\Lambda_2 = \Z_p\llbracket \Gal(K_\infty/K)\rrbracket }$, then 
  \[ \rg_{\Lambda_1}(Y_{A}^{(K_\infty^{(1)})}) = \rg_{\Lambda_1}((Y_{A}^{(K_\infty)})_{\Gamma}), \] 
  and the latter rank is zero for all but finitely many $K_\infty^{(1)} \in \mathcal{E}^{\subseteq K_\infty}(K)$. 
  
  Now we proceed via induction. Suppose that the statement of the theorem is true for all $\Z_p^{l-1}$-extensions, $l \ge 3$. Let $K_\infty/K$ be a $\Z_p^l$-extension, let ${K_\infty^{(1)} \in \mathcal{E}^{\subseteq K_\infty}(K)}$ and $n \in \N$ be arbitrary. It follows from $r_{A}^{(K_\infty)} = 0$ that $r_{A}^{(K_\infty^{(l-1)})} = 0$ for all but finitely many $\Z_p^{l-1}$-extensions of $K$ contained in $K_\infty$ (here $r_{A}^{(K_\infty^{(l-1)})}$ means the rank over the appropriate group ring $\Z_p\llbracket \Gal(K_\infty^{(l-1)}/K)\rrbracket $, respectively). We fix such an extension $K_\infty^{(l-1)}$ which in addition contains the unique subfield ${K_n^{(1)} \subseteq K_\infty^{(1)}}$ of exponent $p^n$ over $K$. Proposition~\ref{lemma:mu_going_down} implies that we may moreover assume that the $e_i$ from the elementary $\Z_p\llbracket \Gal(K_\infty^{(l-1)}/K)\rrbracket $-module attached to $Y(K_\infty^{(l-1)})$ are equal to the $e_i$ from $Y(K_\infty)$, and therefore the hypothesis $k > \max_i e_i$ is still valid. Then the inductive hypothesis implies that we can choose 
  \[ \tilde{K}_\infty^{(1)} \in \mathcal{E}(K_\infty^{(1)},n) \cap \mathcal{E}^{\subseteq K_\infty^{(l-1)}}(K)\] 
  such that $r_{A}^{(\tilde{K}_\infty^{(1)})} = 0$. 
  
\end{proof} 

In the next result, we prove implications between the weak Leopoldt conjecture and the vanishing of the $\mu$-invariant (both for ideal class groups and fine Selmer groups) on the one hand and the property of being a finitely generated $\Z_p\llbracket H\rrbracket $-module for some subgroup $H \subseteq G$ of $\Z_p$-rank $l-1$ on the other hand -- recall that this property has been the main subject in Theorem~\ref{thm:z_p[[H]]}. 
\begin{thm} \label{thm:Conj_A} 
Let $A$ be an abelian variety defined over $K$.
  Let $K_{\infty}/K$ be a $\Z_p^l$-extension and let $\Sigma$ be as in Theorem \ref{thm:comparing ranks}. We suppose that $K(A[p])$ is a finite $p$-extension of $K$; if $l\ge 2$ then we assume that this extension is trivial. If $p = 2$, then we assume that $K$ is totally imaginary. We write $G = \Gal(K_\infty/K)$. 
  
  Then the following statements are equivalent: \begin{compactenum}[(i)] 
    \item $r_{A}^{(K_\infty)} = \mu(Y_{A}^{(K_\infty)}) = 0$, 
    \item there exists a subgroup $H$ such that $G/H\cong \Z_p$ with the property that  $Y^{(K_{\infty})}_A$ is finitely generated over $\Z_p\llbracket H\rrbracket $, 
    \item $\mu(Y^{(K_\infty)}) = 0$. 
  \end{compactenum} 
\end{thm}
\begin{proof}
Let $K'=K(A[p])$ and $K_\infty' = K_\infty K'$. 
Suppose first that ${l=1}$ and ${K' \ne K}$. Since $K'/K$ is a $p$-extension by assumption, Lemma~\ref{lemma:mu_iwasawa_allgemein} implies that ${\mu(Y(K_\infty))>0}$ if and only if ${\mu(Y(K_\infty'))>0}$ (note that both Iwasawa modules are torsion in view of our decomposition hypothesis, as in Theorem~\ref{thm:comparing ranks}). One can further show that $Y^{(K_{\infty})}_A$ is finitely generated over $\Z_p\llbracket H\rrbracket $ if and only if $Y^{(K'_{\infty})}_A$ is finitely generated over $\Z_p\llbracket H\rrbracket $ (compare with the proof of \cite[Proof of Theorem 5.5]{lim-murty}). For the same reason, (i) holds for $Y_A^{(K_\infty)}$ if and only if it holds for $Y_A^{(K_\infty')}$. It therefore suffices to consider the case $K=K'$. 

First note that $r^{(K_\infty)} = \rg_{\Z_p\llbracket G\rrbracket }(Y^{(K_\infty)}) = 0$ by \cite[Theorem~6.1]{ochi-venjakob}, since no prime above $p$ splits completely in $K_\infty/K$ by assumption. Therefore the equivalence of (i) and (iii) follows from Theorem~\ref{thm:comparing ranks}. 

Now let $l \ge 1$ be arbitrary. 
By Theorem \ref{thm:perbet+} we obtain that the quantity $r_p(Y^{(K_{\infty})}_{A}/(\omega_n(T_1),\dots,\omega_n(T_l))Y^{(K_{\infty})}_{A})$ is of size $\Ok(p^{n(l-1)})$ if and only if both $r_{A}^{(K_\infty)}$ and $\mu(Y_{A}^{(K_\infty)})$ are zero. Here ${\omega_n(T_i) = (T_i+1)^{p^n}-1}$ for each $i$. 
This is in turn the case iff $Y_{A}^{(K_\infty)}$ is $\Z_p\llbracket G\rrbracket $-torsion and the characteristic ideal of $Y_{A}^{(K_\infty)}$ is not divisible by $p$. 
That is true if and only if there exists an annihilator $g$ of $Y_{A}^{(K_\infty)}$ which is not divisible by $p$. If the annihilator $g$ of $Y^{(K_\infty)}_A$ is coprime to $p$, we can use \cite[Lemma 1]{babaichev} to choose topological generators $\tau_i$ of $G$, $1 \le i \le l$, such that $g$ is, as an element in $\Z_p\llbracket T_1,\dots, T_l\rrbracket $, represented by a polynomial 
${\sum_{i=0}^{v-1} a_i(T_1,\dots T_{l-1})T_l^i+T_l^v}$ in $T_l$ for suitable power series $a_i(T_1, \ldots, T_{l-1})$ in $\Z_p\llbracket T_1,\ldots, T_{l-1}\rrbracket $ (here $T_i = \tau_i - 1$ for every $i$). Then $g$ is called \emph{regular} with respect to the variable $T_l$. 

If such an annihilator exists, then $Y_{A}^{(K_{\infty})}$ is finitely generated over the Iwasawa algebra  $\Z_p\llbracket T_1,\ldots,T_{l-1}\rrbracket $ and we can choose $H$ as the subgroup generated by $\tau_1,\dots, \tau_{l-1}$. If, conversely, $Y_{A}^{(K_\infty)}$ is finitely generated over $\Z_p\llbracket H\rrbracket $ for some subgroup $H$ such that $G/H\cong \Z_p$, then we can choose topological generators $\tau_1,\ldots, \tau_l$ of $G$ with ${H = \langle \tau_i \mid 1\le i\le l-1 \rangle}$. Then $Y_{A}^{(K_\infty)}$ is a torsion $\Lambda_l$-module and has an annihilator which is a monic polynomial in $(\Z_p\llbracket T_1, \ldots, T_{l-1}\rrbracket )[T_l]$, as above, and therefore is coprime to $p$.
\end{proof}
\begin{rem}
Note that the equivalence of the points (i) and (iii) holds also for more general $p$-adic Lie extensions as a consequence of Theorem~\ref{thm:comparing ranks}. Furthermore, it is easy to check that (ii) implies (i) for more general $p$-adic Lie extensions. To prove the equivalence to point (ii) we need Baba\u{\i}cev's work to choose a distinguished variable $T_l$. As these results seem to be not available for more general groups, the authors currently are not able to prove a more general version of the full theorem.
\end{rem}
\begin{rem}
Coates and Sujatha originally conjectured that in the case of an elliptic curve $A = E$, $Y^{(K_{\infty})}_A$ should be finitely generated over $\Z_p$ for the cyclotomic $\Z_p$-extension $K_{\infty}^{c}$ (see Conjecture~A in \cite{coates-sujatha}). They also proved the equivalence of assertions (i) and (iii) from Theorem~\ref{thm:Conj_A} for the cyclotomic $\Z_p$-extension, $p \ne 2$, in this case. More generally, Lim and Murty proved a special instance of this equivalence for abelian varieties over the cyclotomic $\Z_p$-extension (see \cite[Theorem~5.5]{lim-murty}). 
\end{rem} 

\begin{rem} Recall Greenberg's topology on the set of $\Z_p$-extensions $\mathcal{E}(K)$ of a fixed number field $K$, as introduced at the beginning of this subsection. 
Assuming that the conjecture from the previous remark is true, we can find a dense (with regard to Greenberg's topology) subset $\mathcal{W}$ of $\Z_p$-extensions of $K$ such that $Y^{(K_{\infty})}_A$ is finitely generated over $\Z_p$ for all $K_{\infty}\in \mathcal{W}$, because the triviality of $\mu(K_\infty^c/K)$ implies that a Greenberg dense subset of $\mathcal{E}(K)$ has vanishing $\mu$-invariant (see \cite[Theorem~4]{babaicev}). Similar questions have been addressed also in \cite{non-torsion}. 
\end{rem}

If the $\Z_p^l$-extension $K_\infty/K$ contains $A[p^\infty]$, then we can derive comparison results which are finer than Theorem~\ref{thm:comparing ranks}. In order to do so, we first need an analogue of our $p^n$-version of the class number formula due to Cuoco and Monsky (see Theorem~\ref{thm:ckmm} and Theorem~\ref{thm:vergleich_l0} below). Note that the condition that $A[p^\infty]\subseteq A(K_\infty)$ is quite restrictive. It allows only the following type of abelian varieties.
\begin{lemma} \label{lemma:product_CM} 
Let $A$ be an abelian variety defined over $K$ and assume that $$K(A[p^\infty])/K$$ is a $\Z_p^l$-extension. Then $A$ is isogenous to a product of abelian varieties with complex multiplication. 
\end{lemma}
\begin{proof}
Let $A_1,\ldots, A_r$ be simple abelian varieties such that $A$ is isogenous to 
\[ A_1\times\cdots \times A_r\] 
over $K$. By hypothesis, each $K(A_i[p^\infty])/K$ is an abelian extension, and by the main result of \cite{zarhin} each of the $A_i$ has complex multiplication. 
\end{proof}
We will also frequently assume that all the primes of bad reduction lie above $p$. Using the above lemma together with Theorem~\ref{thm:serre-tate},  this can be obtained by taking a finite extension of the base field $K$. 
\begin{thm} \label{thm:vergleich_l0} 
  Let $A$ be an abelian variety of dimension $d$ defined over $K$, and let $K_\infty/K$ be a $\Z_p^l$-extension which contains $A[p^\infty]$. We assume that $A[p] \subseteq A(K)$, that $A$ has good reduction outside of $p$, $p \ge 3$, and that the decomposition subgroup ${D_v \subseteq Gal(K_\infty/K)}$ has dimension at least 2 for each $v \in \Sigma_p$. 
  
  Let $l \in \N$ be arbitrary. Then 
  \[ |v_p(|(Y_{A}^{(K_\infty)})_{G_n}/p^{n}(Y_{A}^{(K_\infty)})_{G_n}|) - (2dm_0 p^{ln} + 2d l_0 n p^{(l-1)n})| = \Ok(p^{(l-1)n}), \] 
  where $m_0, l_0 \in \N$ are the generalised Iwasawa invariants of the Iwasawa module $Y(K_\infty)$ of ideal class groups. In particular, $Y_{A}^{(K_\infty)}$ is a $\Lambda$-torsion module.
\end{thm} 
\begin{proof} 
  Since $A[p] \subseteq K$ and $A[p^\infty] \subseteq K_\infty$, it follows from (the proof of) Corollary~\ref{cor:pntorsion} that $A[p^n] \subseteq K_n$ for every $n \in \N$. Therefore the proof of Theorem~\ref{newrankthm} (with $\Sigma=\Sigma_p$) implies that 
  \[ r_p^{n}(\Sel_{0,A[p^{n}]}(K_{n})) = r_p^{n}(\textup{Hom}(Y_{\Sigma_p}(K_{n}), A[p^{n}])), \] 
  and the latter equals $2d r_p^{n}(Y_{\Sigma_p}(K_{n}))$. 
  
  Further, Lemma~\ref{lemma:lim-murty} and the hypothesis on the ranks of the decomposition subgroups of the primes above $p$ imply that 
  \[ |v_p(|\Sel_{0,A}(K_n)[p^{n}]|) - v_p(|\Sel_{0, A[p^{n}], \Sigma}(K_n)|)| = \Ok(n p^{(l-2)n}). \] 
  Now we apply the Control Theorem~\ref{thm:control1}: $\Sel_{0,A}(K_n)[p^{n}]$ is the Pontryagin dual of the quotient $Y_{A}^{(K_n)}/p^{n} Y_{A}^{(K_n)}$. By the control theorem, we have 
  \[ |\tilde{r}_p^n((Y_{A}^{(K_\infty)})_{G_n}) - \tilde{r}_p^n(Y_{A}^{(K_n)})| = \Ok(p^{(l-1)n}). \] 
  Since 
  \[ |r_p^n(Y_\Sigma(K_n)) - r_p^n(Y(K_n))| = \Ok(p^{(l-1)n})\] 
  by Lemma~\ref{lemma:vergleich_sigma_cl}, we may conclude that 
  \[ |v_p(|(Y_{A}^{(K_\infty)})_{G_n}/p^{n}(Y_{A}^{(K_\infty)})_{G_n}|) - 2d r_p^n(Y(K_n))| = \Ok(p^{(l-1)n}). \] 
  The assertion now follows from Theorem~\ref{thm:ckmm}. 
\end{proof} 
\begin{rem} 
  In view of Theorem~\ref{thm:CM-suff}, abelian varieties with complex multiplication are a family of abelian varieties for which a $\Z_p^l$-extension $K_\infty/K$ with $A[p^\infty] \subseteq A(K_\infty)$ does exist canonically (see also Lemma~\ref{lemma:product_CM}). Such abelian varieties have potentially good reduction everywhere by Theorem~\ref{thm:serre-tate}. Moreover, since $K_\infty$ automatically contains the cyclotomic $\Z_p$-extension of $K$ (see \cite[Corollary~4.8]{ochi-venjakob2}), the decomposition hypothesis from Theorem~\ref{thm:vergleich_l0} is not very restrictive in general. 
\end{rem} 
Now we extract information on arithmetic invariants of $Y_{A}^{(K_\infty)}$. 
\begin{cor} 
  In the situation of Theorem~\ref{thm:vergleich_l0}, we have ${\rg_{\Z_p\llbracket G\rrbracket }(Y_{A}^{(K_\infty)}) = 0}$ and $\mu(Y_{A}^{(K_\infty)}) = 2d m_0$. 
\end{cor} 
\begin{proof} 
   It follows from Corollary~\ref{cor:perbet+} that the difference between 
   $$v_p(|Y_{A}(K_\infty)_{G_n}/p^{n}Y_{A}(K_\infty)_{G_n}|)$$ and 
   $$(\rg_{\Z_p\llbracket G\rrbracket }(Y_{A}(K_\infty)) n + \mu(Y_{A}(K_\infty))) p^{ln}$$ 
   is $\Ok(n p^{(l-1)n})$. 
   Now we apply Theorem~\ref{thm:vergleich_l0} and compare the leading terms. 
\end{proof} 

While the above corollary could have been derived also from Theorem~\ref{thm:comparing ranks}, the next result really uses the hypothesis that $A(K_\infty)$ contains all of $A[p^\infty]$ (cf. also Theorem~\ref{thm:D} from the \nameref{section:1}). 
\begin{thm} \label{thm:l0absch} 
  Let $A$ be an abelian variety defined over $K$, and let $K_\infty/K$ be a $\Z_p^l$-extension which contains $A[p^\infty]$. We assume that $A$ has good reduction outside of $p$, $p \ge 3$, and that the decomposition subgroup ${D_v \subseteq \Gal(K_\infty/K)}$ has dimension at least 2 for each $v \in \Sigma_p$. 
  Let $L/K$ be a $\Z_p$-extension which is contained in $K_\infty$, and write ${H = \Gal(K_\infty/L)}$. 
  Suppose that $Y_{A}^{(K_\infty)}$ is finitely generated over $\Z_p\llbracket H\rrbracket $, and that ${\mu_{\Z_p\llbracket H\rrbracket }(Y_{A}^{(K_\infty)}) = 0}$. Then 
  \[ 2d l_0(Y(K_\infty)) \le \rg_{\Z_p\llbracket H\rrbracket }(Y_{A}^{(K_\infty)}). \] 
\end{thm} 
We note that $\mu(Y_{A}^{(K_\infty)}) = 0$ by the assumption on $Y_{A}^{(K_\infty)}$. 
\begin{proof} In this proof, all fine Selmer groups are defined with respect to the finite set $\Sigma=\Sigma_p$. 
  It follows from a result of Harris (see Proposition~\ref{prop:harris}) that 
  \[ |\rg_{\Z_p}((Y_{A}^{(K_\infty)})_{H_n}) - \rg_{\Z_p\llbracket H\rrbracket }(Y_{A}^{(K_\infty)}) p^{(l-1)n}| = \Ok(p^{(l-2)n}).\] 
  Our hypotheses imply that 
  \[ | v_p(|(Y_{A}^{(K_\infty)})_{H_n}/p^{n} (Y_{A}^{(K_\infty)})_{H_n}|) - n \cdot \rg_{\Z_p}((Y_{A}^{(K_\infty)})_{H_n})| = \Ok(np^{n(l-2)}). \]
  Indeed, we have an exact sequence 
  \[ 0 \longrightarrow X \longrightarrow Y_{A}^{(K_\infty)} \longrightarrow \Z_p\llbracket H\rrbracket ^r \oplus \bigoplus_{j = 1}^t \Z_p\llbracket H\rrbracket /(g_j), \] 
  where $X \subseteq Y_{A}^{(K_\infty)}$ denotes the maximal pseudo-null $\Z_p\llbracket H\rrbracket $-submodule, $$r = \rg_{\Z_p\llbracket H\rrbracket }(Y_{A}^{(K_\infty)})$$ 
  and where the $g_j \in \Z_p\llbracket H\rrbracket $ are coprime with $p$. We will show that 
  \[ \rg_p((Y_{A}^{(K_\infty)})_{H_n}[p^\infty]) = \Ok(p^{(l-2)n}). \] 
  First, it is clear that $(\Z_p\llbracket H\rrbracket )_{H_n}$ is $\Z_p$-free for every $n \in \N$. As in the proof of Theorem~\ref{thm:Conj_A}, the topological generators of $G \supseteq H$ can be chosen such that each $g_j$ is \emph{regular} with respect to some $T_i$, $2 \le i \le l$; therefore $$\rg_p((\Z_p\llbracket H\rrbracket /(g_j))_{H_n}) = \Ok(p^{(l-2)n}).$$ 
  Finally, it follows from \cite[Lemma~3.1]{cuoco-monsky} that ${\rg_p(X_{H_n}) = \Ok(p^{(l-2)n})}$. 
  
  Therefore  
  \[ | v_p(|(Y_{A}^{(K_\infty)})_{H_n}/p^{n} (Y_{A}^{(K_\infty)})_{H_n}|) - n \cdot \rg_{\Z_p\llbracket H\rrbracket }(Y_{A}^{(K_\infty)}) p^{(l-1)n}| = \Ok(n p^{(l-2)n}). \] 
  Since $v_p(|(Y_{A}^{(K_\infty)})_{G_n}/p^{n} (Y_{A}^{(K_\infty)})_{G_n}|) \le v_p(|(Y_{A}^{(K_\infty)})_{H_n}/p^{n} (Y_{A}^{(K_\infty)})_{H_n}|)$, the assertion follows by comparing the leading terms of this inequality and Theorem~\ref{thm:vergleich_l0} (it is here where we use ${\Sigma = \Sigma_p}$). 
\end{proof} 
\begin{rem} 
  Let $\gamma$ be a topological generator of $G/H \cong \Z_p$. If $T = \gamma - 1$ happens to be an annihilator of $Y_{A}^{(K_\infty)}$, then taking $G_n$-coinvariants is the same as taking $H_n$-invariants for all $n \in \N$, i.e. in this case the inequality in Theorem~\ref{thm:l0absch} can be replaced by \lq$=$'. 
\end{rem} 

Recall that we have proved also a result comparing the $\Z_p\llbracket H\rrbracket $-ranks of $Y_{A}^{(K_\infty)}$ and $Y(K_\infty)$ if they are finite (cf. Theorem~\ref{thm:z_p[[H]]}). In our final comparison result, which corresponds to Theorem~\ref{thm:C} from the \nameref{section:1}, we ask for a direct relation between $l_0(Y(K_\infty))$ and $l_0(Y_{A}^{(K_\infty)})$ -- the latter can be defined as  $\Gal(K_\infty/K) \cong \Z_p^l$. 
\begin{thm} \label{thm:l0_Gleichheit} 
Keep the assumptions from Theorem~\ref{thm:vergleich_l0}.
 Then 
  \[ l_0(Y_{A}^{(K_\infty)}) = 2d l_0(Y(K_\infty)). \] 
\end{thm} 
\begin{proof} 
We need the following purely algebraic generalisation of the work of Cuoco and Monsky (see also our corresponding results in Section~\ref{section:5}).

Suppose first that $Y=\Z_p\llbracket G\rrbracket /(F^s)$ for some \emph{special} irreducible polynomial $F$ in the sense of \cite{cuoco-monsky} and $s \in \N$. Then $Y_{G_n}$ is  $\Z_p$-torsionfree and has $\Z_p$-rank $l_0p^{n(l-1)}$ for $n$ large enough, by (the proof of) \cite[proof of Theorem 4.13]{cuoco-monsky}. 
Therefore $v_p(|Y_{G_n}/p^n Y_{G_n}|) = n l_0 p^{n(l-1)}$ for sufficiently large $n \in \N$. 

Let now $X$ be an arbitrary $\Z_p\llbracket G\rrbracket $-torsion module. Let $E_1\oplus E_2$ be the corresponding elementary module, where $E_2$ is annihilated by a product of special primes and $E_1$ corresponds to all other annihilators. For any module $M$ such that $M_{G_n}$ is of $\Z_p$-rank $\mathcal{O}(p^{n(l-2)})$, we have that 
\[\vert v_p(M_{G_n}/p^n M_{G_n})-v_p(M_{G_n}[p^\infty]/p^nM_{G_n}[p^\infty])\vert=\mathcal{O}(np^{n(l-2)}).\]
In particular, for pseudo-null modules, we have that $v_p(\vert M_{G_n}/p^nM_{G_n})=\mathcal{O}(p^{n(l-1)})$ (see \cite[Theorem 3.2]{cuoco-monsky}. 

Now we use the following variant of Lemma~\ref{lemma:ckmm1} (note that ${M_{G_n} = M_{I_n}}$ in the notation from Section~\ref{section:5}): 
\begin{lemma} 
   Fix a $\Z_p^l$-extension $K_\infty/K$ with Galois group $G$, $l \ge 1$. Let $M$ and $M'$ be two finitely generated and torsion $\Z_p\llbracket G\rrbracket $-modules, and suppose that $M$ and $M'$ are pseudo-isomorphic. Then 
   \[ |v_p(|M_{G_n}/p^n M_{G_n}|) - v_p(|M'_{G_n}/p^n M'_{G_n}|)| = \Ok(p^{n(l-1)}). \] 
\end{lemma} 
\begin{proof} 
  Let $M_1$ and $M_2$ be the kernel and cokernel of a pseudo-isomorphism ${M\longrightarrow M'}$. Breaking up the exact sequence
\[0\longrightarrow M_1\longrightarrow M\longrightarrow M' \longrightarrow M_2 \longrightarrow 0\]
into two sequences 
\begin{align*}
    0 \longrightarrow M_1\longrightarrow M\longrightarrow H\longrightarrow 0, \\
    0\longrightarrow H\longrightarrow M'\longrightarrow M_2\longrightarrow 0, 
\end{align*}
and taking $G_n$-coinvariants, we obtain
\begin{align*}
    (M_1)_{G_n}\longrightarrow M_{G_n}\longrightarrow H_{G_n}\longrightarrow 0, \\
    H_{G_n}\longrightarrow (M')_{G_n}\longrightarrow (M_2)_{G_n}\longrightarrow 0. 
\end{align*}
For both sequences we take the $p^n$-quotients. As the sequences are only right exact, we obtain
\begin{align*}
    (M_1)_{G_n}/p^n(M_1)_{G_n}\longrightarrow M_{G_n}/p^nM_{G_n}\longrightarrow H_{G_n}/p^nH_{G_n}\longrightarrow 0, \\
    H_{G_n}/p^nH_{G_n}\longrightarrow (M')_{G_n}/p^n(M')_{G_n}\longrightarrow (M_2)_{G_n}/p^n(M_2)_{G_n}\longrightarrow 0. 
\end{align*}
We get the two inequalities \[v_p(\vert H_{G_n}/p^nH_{G_n}\vert )\le v_p(\vert M_{G_n}/p^nM_{G_n}\vert)\] and \begin{align*}
v_p(\vert (M')_{G_n}/p^n(M')_{G_n}\vert )&\le v_p(\vert H_{G_n}/p^nH_{G_n}\vert)+v_p(\vert (M_2)_{G_n}/p^n(M_2)_{G_n}\vert) \\&=v_p(\vert H_{G_n}/p^nH_{G_n}\vert)+\mathcal{O}(p^{n(l-1)}).\end{align*} So in total we have 
\[v_p(\vert (M')_{G_n}/p^n(M')_{G_n}\vert ) =  v_p(\vert M_{G_n}/p^nM_{G_n}\vert)+\mathcal{O}(p^{n(l-1)}).\]
Interchanging the roles of $M$ and $M'$, we obtain 
\[\vert v_p(\vert M_{G_n}/p^n M_{G_n}\vert)-v_p(|M'_{G_n} / p^n M'_{G_n}|)\vert=\mathcal{O}(p^{n(l-1)}).\]
\end{proof} 
Applying this lemma to $X$ and $E_1 \oplus E_2$, we are left to compute $v_p(\vert N_{G_n}/p^nN_{G_n}\vert)$ for the elementary modules $N=E_1$ and $N=E_2$. For $E_1$, this is done in Lemma \ref{lemma:ckmm2}, and for $E_2$ we get $nl_0(E_2)p^{n(l-1)}$ by the above. So we obtain that 
\[v_p(\vert X_{G_n}/p^nX_{G_n}\vert)=m_0p^{ln}+l_0np^{n(l-1)}+\mathcal{O}(p^{n(l-1)}).\]
   
Applying this to $Y_{A}^{(K_\infty)}$ (defined with respect to ${\Sigma = \Sigma_p}$), we obtain
   \[ v_p(|(Y_{A}^{(K_\infty)})_{G_n} /p^{n} (Y_{A}^{(K_\infty)})_{G_n}|) = (\mu(Y_{A}^{(K_\infty)}) p^n + l_0(Y_{A}^{(K_\infty)}) n  + \Ok(1)) p^{(l-1)n}. \] 
   
   Note that $\mu(Y_A^{(K_\infty)}) = m_0(Y(K_\infty))$ in view of Theorem
   \ref{thm:comparing ranks}. The assertion follows by comparing the above asymptotic formula with that from Theorem~\ref{thm:vergleich_l0}. 
\end{proof}

\bibliography{references} 
 	
\bibliographystyle{alpha}

\end{document}